\documentclass[final]{siamltex}

\pdfoutput=1

\usepackage{amsmath}
\usepackage{amsfonts}
\usepackage{graphicx}
\usepackage{url}
\usepackage{subfig}

\newcommand{\sJ}{\mathcal{J}}
\newcommand{\sS}{\mathcal{S}}

\newcommand{\E}{\mathrm{E}}
\newcommand{\cov}{\mathrm{cov}}

\newcommand{\ddt}[1]{#1^\prime}
\newcommand{\bmat}[1]{\begin{bmatrix}#1\end{bmatrix}}

\newcommand{\mini}[1]{\underset{#1}{\operatorname{minimize}}}
\newcommand{\minn}[1]{\underset{#1}{\operatorname{min}}}
\newcommand{\maxi}[1]{\underset{#1}{\operatorname{maximize}}}
\newcommand{\st}{\text{subject to }}
\newcommand{\smin}{\sigma_{\mathrm{min}}}
\newcommand{\astar}{a}
\newcommand{\ahat}{\hat{a}}
\newcommand{\ahatstar}{\hat{a}}
\newcommand{\ahatonestar}{\hat{a}_{1}}

\newcommand{\ahattwostart}{\hat{a}_{2}^{(t)}}
\newcommand{\lstar}{\lambda}
\newcommand{\lstart}{\lambda^{(t)}}

\title{Residual Minimizing Model Interplation for Parameterized Nonlinear Dynamical Systems}

\author{Paul G.~Constantine\thanks{Department of Mechanical Engineering, Stanford University, Stanford, California 94305
({\tt paul.constantine@stanford.edu}).}
\and 
Qiqi Wang\thanks{Department of Aeronautics and Astronautics,
Massachusetts Institute of Technology, Cambridge, Massachusetts 02139 ({\tt qiqi@mit.edu}).}
}

\begin{document}

\maketitle

\begin{abstract}
We present a method for approximating the solution of a parameterized, nonlinear dynamical system using an
affine combination of solutions computed at other points in the input parameter space. The coefficients of the
affine combination are computed with a nonlinear least squares procedure that minimizes the residual of the governing
equations. The approximation properties of this residual minimizing scheme are comparable to existing reduced basis and
POD-Galerkin model reduction methods, but its implementation requires only independent evaluations of the nonlinear
forcing function. It is particularly appropriate when one wishes to approximate the states at a few points in time
without time marching from the initial conditions. We prove some interesting characteristics of the scheme including an
interpolatory property, and we present heuristics for mitigating the effects of the ill-conditioning and reducing the
overall cost of the method. We apply the method to representative numerical examples from kinetics -- a three state
system with one parameter controlling the stiffness -- and conductive heat transfer -- a nonlinear parabolic PDE with a
random field model for the thermal conductivity.
\end{abstract}


\pagestyle{myheadings}
\thispagestyle{plain}
\markboth{P.~G. CONSTANTINE AND Q.~WANG}{RESIDUAL MINIMIZING MODEL INTERPOLATION}

\section{Introduction}
\label{sec:intro}

As computational capabilities grow, engineers and decision makers increasingly rely on simulation to aid in 
design and decision-making processes. However, the complexity of the models has kept pace with the growth in computing
power, which has resulted in expensive computer models with complicated parametric dependence. For given parameter
values, each costly evaluation of the model can require extensive time on massively parallel, high performance systems.
Thus, exhaustive parameter studies exploring the relationships between input parameters and model outputs become
infeasible; cheaper reduced order models are needed for sensitivity/uncertainty analysis, design optimization, and
model calibration. 

Reduced order modeling has become a very active field of research. Methods based on the Proper Orthogonal
Decomposition (POD) have been shown to dramatically reduce the computational complexity for approximating the solution
of linear dynamical systems or parameterized linear steady state problems~\cite{Antoulas01,Roms08,Prudhomme02,Bui08}.
The success of such methods for linear models has spurred a slew of recent work on model reduction for nonlinear
models~\cite{Carlberg10,Chaturantabut10,Kerfriden10,Galbally10,Nguyen08}.
In this vein, we focus on nonlinear dynamical systems that depend on a set of input parameters, but the method we
develop can be modified for steady and/or linear parameterized models as well. The parameters may affect material
properties, boundary conditions, forcing terms and/or model uncertainties; we do not consider parameterized initial
conditions. Suppose we can afford to compute a full model solution at a few points in the parameter space, or
equivalently, suppose we have access to a database of previously computed runs. How can we use those stored runs to
cheaply estimate the model output at untested input parameter values?

In this paper, we propose an interpolation\footnote{In the work
of~\cite{Barrault04,Chaturantabut10,Nguyen08-2}, interpolation occurs in the spatial domain; our approximation
interpolates in the parameter domain.} method that employs an affine combination of the stored model evaluations and a
nonlinear least squares procedure for computing the coefficients of the affine combination such that the equation
residual is minimized. If one is interested in approximating the state vector at a new parameter value at only a few
points in time, then constructing and solving the least squares problem will be significantly cheaper than solving the
true model; this justifies the comparison to reduced order modeling. However, if one wishes to approximate the full time
history of a few elements of the state vector, then the least squares problem may be more expensive to solve than the
true model. We will highlight the former situation in the second numerical example.

Like standard ODE solvers, the method
only requires evaluations of the forcing function from the dynamical system. Thus, implementation is straightforward
using existing codes. However, unlike ODE solvers, the function evaluations can be performed independently to take
advantage of parallel architectures. The model interpolation scheme itself has many appealing features including
point-wise optimality by construction, reuse of stored model evaluations, and a strategy for adaptively choosing points in the
parameter space for additional full model evaluations. We observe that the error in approximation decays like the
eigenvalues of a covariance-like operator of the process as more model evaluations added. But we show that achieving a
small residual requires the solution of an ill-conditioned least squares problem; we present a heuristic for taming the
potentially unwieldy condition number.

The linear model of the data is a common feature of many model reduction and interpolation schemes including reduced
basis methods~\cite{Prudhomme02,Veroy05,Grepl05,Bui08}, kriging surfaces/Gaussian process
emulators~\cite{Stein99,Cressie93,Kennedy01}, and polynomial interpolation schemes~\cite{Barthelmann00,Xiu05}. Unlike
the kriging and polynomial schemes, our process utilizes the governing equations to construct the minimization problem
that yields the coefficients, which invariably produces a more accurate interpolant; this is comparable to
the reduced basis methods and POD-based model reduction techniques~\cite{Chaturantabut10,Antoulas01}. However, in
contrast to reduced basis methods and POD-based techniques, we formulate the least squares problem using only the
equation residual, which requires minimal modifications to the full system solvers. Additionally, this residual-based
formulation applies directly to nonlinear models, which have posed a persistent challenge for schemes
based on Galerkin projection.

In Section \ref{sec:algorithm}, we pose the model problem and derive the residual minimizing scheme, including some
specific details of the nonlinear least squares solver. We briefly show in Section \ref{sec:comp} how kriging
interpolation relates to the least squares procedure. In Section \ref{sec:analysis}, we prove
some properties of the residual minimizing scheme, including a lower bound on the average error, an interpolatory
property, and some statements concerning convergence. Section \ref{sec:heuristics} presents heuristics
for adding full model evaluations, reducing the cost of the scheme, and managing the ill-conditioning in the least
squares problems. In Section \ref{sec:examples}, we perform two numerical studies: (i) a simple nonlinear
dynamical system with three state variables and one parameter controlling the stiffness, and (ii) a
two-dimensional nonlinear parabolic PDE model of heat transfer with a random field model for thermal conductivity.
Finally we conclude in Section \ref{sec:conclusions} with a summary and directions for future work.

\section{Residual Minimizing Model Interpolation}
\label{sec:algorithm}

Let $x(t)=x(t,s)$ be a $\mathbb{R}^p$-valued process that satisfies the dynamical system
\begin{equation}
\label{eq:model}
\ddt{x} = f(x,t,s), \qquad t\in[0,T],\qquad s\in\sS,
\end{equation}
with initial condition $x(0)=x_0$. The space $\sS\subset\mathbb{R}^d$ is the input parameter space, and we assume
the process is bounded for all $s\in\sS$ and $t\in[0,T]$. Note that the parameters affect only the dynamics of the
system -- not the initial condition. The $\mathbb{R}^p$-valued function $f$ is nonlinear in the states $x$, and it
often represents a discretized differential operator with forcing and boundary terms included appropriately. We assume
that $f$ is Lipschitz continuous and its Jacobian with respect to the states is nonsingular, which excludes systems with
bifurcations. The process satisfying \eqref{eq:model} is unique in the following sense: Let $y(t)$ be a differentiable
$\mathbb{R}^p$-valued process with $y(0)=x_0$, and define the integrated residual $\phi(y)$ as
\begin{equation}
\label{eq:resid}
\phi(y) \;=\; \int_0^T \left\|\ddt{y}-f(y,t,s)\right\|^2\,dt,
\end{equation}
where $\|\cdot\|$ is the standard Euclidean norm. If $\phi(y)=0$, then clearly $y(t)=x(t)$ for $t\in[0,T]$. This
residual-based definition of uniqueness underlies the model interpolation method. Given a discretization of the time
domain $0\leq t_1<\cdots<t_m\leq T$, we define the discretized residual $\phi_m(y)$ as
\begin{equation}
\label{eq:dresid}
\phi(y) \;\approx\;
\phi_m(y) \;=\; \sum_{i=1}^m w_i^2\,\left\|\ddt{y}(t_i)-f(y(t_i),t_i,s)\right\|^2,
\end{equation}
where $w_i^2$ is the integration weight associated with time $t_i$; we write the weights as squared quantities to avoid
the cumbersome square root signs later on. Note that the time discretization of \eqref{eq:dresid} may be a subset of the
time grid used to integrate $x(t)$. In fact, the $t_i$ may be chosen to select a small time window within $[0,T]$. 
Again, it is clear that if $\phi_m(y)=0$, then $y(t_i)=x(t_i)$.

In what follows, we describe a method
for approximating the process $x(t)=x(t,s)$ using a set of solutions computed at other points in the input parameter
space. The essential idea is to construct an affine combination of these precomputed solutions, where the coefficients
are computed with a nonlinear least squares minimization procedure on the discretized residual $\phi_m(y)$. It may not
be immediately obvious that the approximation interpolates the precomputed solutions; we will justify the label
\emph{interpolant} with Theorem \ref{thm:interp}.

Let $x_j(t)=x(t,s_j)$ be the time dependent process with input parameters $s_j\in\sS$ for $j=1,\dots,n$; these represent
the precomputed evaluations of \eqref{eq:model} for the input parameters $s_j$, and we refer to them as the \emph{bases}.
To approximate the process $x(t)=x(t,s)$ for a given $s$, we seek constants $a_j=a_j(s)$ that are independent of time such
that
\begin{equation}
\label{eq:linmodel}
x(t) \;\approx\; \tilde{x}(t) \;=\; \sum_{j=1}^n a_j x_j(t) \;\equiv\; X(t)a,
\end{equation}
where $a$ is an $n$-vector whose $j$th element is $a_j$, and $X(t)$ is a time dependent matrix whose $j$th
column is $x_j(t)$. By definition, the coefficients of the affine combination satisfy
\begin{equation}
\label{eq:constraint}
1\;=\; \sum_{j=1}^n a_j \;=\; e^Ta,
\end{equation}
where $e$ is an $n$-vector of ones. This constraint ensures that the approximation exactly reproduces components of
$x(t)$ that do not depend on the parameters $s$. To see this, let $\chi(t)$ be some component of the state
vector $x(t,s)$ that is independent of $s$. Then
\begin{equation}
\sum_{j=1}^n a_j \chi(t) \;=\; \chi(t)\,\left(\sum_{j=1}^n a_j\right) \;=\; \chi(t).
\end{equation}
For example, \eqref{eq:constraint} guarantees that
the approximation satisfies parameter independent boundary conditions, which arise in many applications of interest.
Since the coefficients $a$ are independent of time, the time derivative of the approximation can be computed as
\begin{align*}
\ddt{\tilde{x}} &= \sum_{j=1}^n a_j \ddt{x_j}\\
&= \sum_{j=1}^n a_j f(x_j(t),t,s_j)\\
&= F(t)a,
\end{align*}
where $F(t)$ is a time dependent matrix whose $j$th column is $f(x_j,t,s_j)$. Note that this can be modified to include
a constant or time dependent mass matrix, as well.

With an eye toward computation, define the matrices $X_i=X(t_i)$ and $F_i=F(t_i)$. Then the discretized residual
$\phi_m(\tilde{x})$ becomes
\begin{align*}
\phi_m(\tilde{x}) 
&= \sum_{i=1}^m w_i^2\,\|\ddt{\tilde{x}}(t_i) - f(\tilde{x}(t_i),t_i,s)\|^2\\
&= \sum_{i=1}^m w_i^2\,\|F_ia - f(X_ia,t_i,s)\|^2\\
&\equiv \rho(a)
\end{align*}
To compute the coefficients $a$ of the approximation, we solve the nonlinear least squares problem
\begin{equation}
\label{eq:nlls}
\begin{array}{ll}
\mini{a} & \rho(a) \\
\st       & e^Ta = 1.
\end{array}
\end{equation}
Note that the minimizer of \eqref{eq:nlls} may not be unique due to potential nonconvexity in $f$, but this should not
deter us. Since we are interested in approximating $x(t)$, any minimizer -- or near minimizer -- of \eqref{eq:nlls} will
be useful.

One could use a standard optimization routine~\cite{Nocedal06} for a nonlinear objective with linear equality
constraints to solve \eqref{eq:nlls}. However, we offer some specifics of a nonlinear least squares algorithm tuned to
the details of this particular problem, namely (i) a single linear equality constraint representing the sum of the
vector elements, (ii) the use of only evaluations of the forcing function to construct the data of the minimization
problem, and (iii) the ill-conditioned nature of the problem.

\subsection{A Nonlinear Least Squares Solver}
To simplify the notation, we define the following quantities:
\begin{equation}
\label{eq:Fphi}
F =
\bmat{w_1 F_1\\ \vdots \\ w_m F_m}
\qquad
\varphi(a) = 
\bmat{
w_1 f(X_1a,t_1,s)\\ \vdots \\ w_m f(X_ma,t_m,s)
}
\end{equation}
The residual $h:\mathbb{R}^n\rightarrow\mathbb{R}^{mp}$ is defined as
\begin{equation}
\label{eq:h}
h(a) = Fa-\varphi(a)
\end{equation}
so that the objective function $\rho(a)$ from \eqref{eq:nlls} can be written as
\begin{equation}
\label{eq:rho}
\rho(a)=\|h(a)\|^2.
\end{equation} 
Let $\sJ=\sJ(a)\in\mathbb{R}^{mp\times n}$ be the Jacobian of $h(a)$. Given a guess $a_k\in\mathbb{R}^n$ such that
$e^Ta_k=1$, the standard Newton step is computed by solving the constrained least squares problem
\begin{equation}
\label{eq:standardnewton}
\begin{array}{ll}
\mini{\delta} & \|\sJ_k\delta + h(a_k)\| \\
\st       & e^T\delta = 0,
\end{array}
\end{equation}
where $\sJ_k=\sJ(a_k)$. Let $\delta_k$ be the minimizer, so that the update becomes
\begin{equation}
\label{eq:update}
a_{k+1} = a_k + \delta_k.
\end{equation}
The constraint on $\delta$ ensures that $a_{k+1}$ sums to one.

Instead of using the standard Newton step \eqref{eq:standardnewton}, we wish to rewrite the problem slightly. Writing it
in this alternative form suggests a method for dealing with the ill-conditioned nature of the problem, which we will
explore in Section \ref{sec:illcon}. We plug the update step \eqref{eq:update} directly into the residual vector and
exploit the constraint $e^Ta_{k+1}=1$ as
\begin{align*}
\sJ_k\delta_k + h(a_k) &= \sJ_k(a_{k+1}-a_k) + h(a_k)\\
&= \sJ_ka_{k+1} + (h(a_k) - \sJ_ka_k)\\
&= \sJ_ka_{k+1} + (h(a_k) - \sJ_ka_k)e^Ta_{k+1}\\
&\equiv R_ka_{k+1}
\end{align*}
where 
\begin{equation}
\label{eq:defRk}
R_k = \sJ_k + (h(a_k) - \sJ_ka_k)e^T.
\end{equation}
Written in this way, each Newton iterate $a_{k+1}$ can be computed by solving the constrained least squares problem
\begin{equation}
\label{eq:newton}
\begin{array}{ll}
\mini{a} & \|R_ka\| \\
\st       & e^Ta = 1.
\end{array}
\end{equation}
This is the heart of the residual minimizing model interpolation scheme. Notice that we used the standard Newton step in
\eqref{eq:update} without any sort of globalizing step length~\cite{Eisenstat94}. To include such a globalizer, we
simply solve \ref{eq:newton} for $a_{k+1}$ and compute $\delta_k$ from \eqref{eq:update}.

We can relate the minimum residual of \eqref{eq:newton} to the true equation residual. Define $r_{k+1}=R_ka_{k+1}$, then
\begin{equation}
\label{eq:residbound}
\|r_{k+1}\| \leq \|\sJ_k\|\|a_{k+1}-a_k\| + \sqrt{\rho(a_k)}.
\end{equation}
Near the solution, we expect $\sJ_k$ to be bounded and the difference in iteration to be small, so that the first term
is negligible. We will use this bound to relate the norm of the equation residual to the conditioning of the constrained
least squares problem in Theorem \ref{thm:illcon}.

\subsubsection{Jacobian-free Newton Step}
\label{sec:fdjacob}

To enable rapid implementation, we next show how to construct a finite difference Jacobian for the nonlinear least
squares using only evaluations of the forcing function $f$ from \eqref{eq:model}. We can write out $\sJ$ as
\begin{equation}
\sJ \;=\; \nabla_a h \;=\; F-\sJ_f X,
\end{equation}
where
\begin{equation}
\sJ_f = 
\bmat{
w_1 \nabla_x f(X_1a,t_1,s) & & \\
 & \ddots &  \\
 & & w_m \nabla_x f(X_ma,t_m,s)
}
\qquad
X = 
\bmat{X_1\\ \vdots \\ X_m}.
\end{equation}
Notice that the Jacobian of $h$ with respect to $a$ contains terms with the Jacobian of $f$ with respect to
the states $x$ multiplied by the basis vectors. Thus, we need only the action of $\sJ_f$ on vectors, similar to
Jacobian-free Newton-Krylov methods~\cite{Knoll04}. We can approximate the action of the Jacobian of $f$ on a vector
with a finite difference gradient; let $x_j$ be the $j$th column of $X$, then
\begin{equation}
\sJ_fx_j\;\approx\;
\frac{1}{\varepsilon}
\left(
\bmat{
w_1 f(X_1a + \varepsilon x_j(t_1),t_1,s)\\
\vdots\\
w_m f(X_ma + \varepsilon x_j(t_m),t_m,s)\\
}
- \bmat{	
w_1 f(X_1a,t_1,s)\\
\vdots\\
w_m f(X_ma,t_m,s)
}\right).
\end{equation}
We can assume that the terms $f(X_ia,t_i,s)$ were computed before approximating the Jacobian to check the norm of the
residual. Therefore, at each Newton iteration we need $nm$ evaluations of $f$ to compute the approximate Jacobian -- one
for each basis at each point in the time discretization. The implementation will use this finite difference
approximation. But for the remainder of the analysis, we assume we have the true Jacobian.

\subsubsection{Initial Guess}
\label{sec:initguess}
The convergence of nonlinear least squares methods depends strongly on the initial guess. In this section, we
propose an initial guess based on treating the forcing function $f$ as though it was linear in the states. To justify
this treatment, let $f(x) = f(x,t,s)$ for given $t$ and $s$. We take the Taylor expansion about the approximation
$\tilde{x}$ as
\begin{equation}
\label{eq:taylor}
f(x) = f(\tilde{x}) + \nabla_xf(\tilde{x}) (x-\tilde{x}) + \dots
\end{equation} 
Evaluate this expansion at $x_j=x_j(t)$, multiply it by $a_j$, and sum over $j$ to get
\begin{align*}
\sum_{j=1}^n a_j f(x_j)
&= \sum_{j=1}^n a_j f(\tilde{x}) + \sum_{j=1}^n a_j \nabla_xf(\tilde{x}) (x_j-\tilde{x}) + \dots\\
&= f(\tilde{x})\underbrace{\left(\sum_{j=1}^na_j\right)}_{=\;1}
+ \nabla_xf(\tilde{x}) \underbrace{\left(\sum_{j=1}^na_jx_j-\tilde{x}\right)}_{=\;0} + \dots
\end{align*}
The constraint eliminates the first order terms in Taylor expansion. If we ignore the higher order terms, then we
can approximate
\begin{equation}
f(\tilde{x},t,s) \;\approx\; \sum_{j=1}^n a_jf(x_j,t,s) \;\equiv\; G(t) a,
\end{equation}
where the $j$th column of $G(t)$ is $f(x_j,t,s)$. Let $G_i=G(t_i)$, and define the $mp\times n$ matrix $G$ as 
\begin{equation}
\label{eq:initg}
G = \bmat{w_1 G_1\\ \vdots \\ w_m G_m}.
\end{equation}
Then to solve for the initial guess $a_0$, we define 
\begin{equation}
\label{eq:rzero}
R_{-1}=F-G 
\end{equation}
and solve \eqref{eq:newton}. In words, we treat $f$ as linear in the states and solve the same constrained
least squares problem. Of course, if $f$ is actually linear, then this is the only step in the approximation; no
Newton iterations are required. In fact, we show in Theorem \ref{thm:initinterp} that using this procedure alone to
compute the coefficients $a$ also produces an interpolant. Therefore, if the system is locally close to linear in the
states, then a small number of Newton iterations will be sufficient.

\section{Comparison to Kriging Interpolation}
\label{sec:comp}

In this section, we compare the quantities computed in the residual minimizing scheme to kriging interpolation, which is
commonly used in geostatistics. 
Suppose that $x(t,s)$ is a scalar valued process with $t\in[0,T]$ and $s\in\sS$.
Following kriging nomenclature, we treat $[0,T]$ as the sample space and $t$ as a random coordinate. Assume for
convenience that $x(t,s)$ has mean zero at each $s\in\sS$,
\begin{equation}
\label{eq:exp}
\E[x(s)] \;=\; \int_0^T x(t,s) \,dt \;=\; 0,\quad s\in\sS.
\end{equation}
The covariance function of the process is then
\begin{equation}
\label{eq:cov}
\cov(s_i,s_j) = \E[x(s_i)x(s_j)].
\end{equation}
Next suppose we are given the fixed values $x_j=x(s_j)$, and we wish to approximate $x=x(s)$ with the
affine model
\begin{equation}
\label{eq:krigmodel}
x\;\approx\;\tilde{x}\;=\;\sum_{j=1}^n x_ja_j,\qquad \sum_{j=1}^n a_j=1.
\end{equation}
This is the model used in ordinary kriging interpolation~\cite{Cressie93}. To compute the coefficients $a_j$, one builds
the following linear system of equations from the assumed covariance function \eqref{eq:cov} (which is typically a model
calibrated with the $\{s_j,x_j\}$ pairs). Define $A_{ij}=\cov(s_i,s_j)$ and $b_i=\cov(s_i,s)$. Then the vector of
coefficients $a$ satisfies
\begin{equation}
\label{eq:krigingkkt}
\bmat{A & e \\ e^T & 0}\bmat{a\\ \lambda}=\bmat{b\\ 1},
\end{equation}
where $\lambda$ is the Lagrange multiplier. 

Next we compare the kriging approach to model interpolation on a scalar valued process. Let $x_j$ be
the vector whose $i$th element is $x(t_i,s_j)$ with $i=1,\dots,m$; in other words, $x_j$ contains the time history of
the process $x$ with input parameters $s_j$. Again, assume that the temporal average of $x(t,s)$ is zero for all
$s\in\sS$. The affine model to approximate the time history for $x=x(s)$ is
\begin{equation}
\label{eq:rmmodel}
x\;\approx\;\tilde{x}\;=\;\sum_{j=1}^n x_ja_j\;=Xa,\qquad 
\sum_{j=1}^n a_j\;=\;e^Ta\;=\;1,
\end{equation}
where $X$ is a matrix whose $j$th column is $x_j$. Notice the similarity between \eqref{eq:krigmodel}
and \eqref{eq:rmmodel}. One way to compute the coefficients $a$ in the spirit of an error minimizing scheme would be to
solve the least squares problem
\begin{equation}
\label{eq:rawlsi}
\begin{array}{ll}
\mini{a} & \|Xa-x\| \\
\st       & e^Ta = 1.
\end{array}
\end{equation}
Ignore for the moment that solving this least squares problem requires $x$ -- the exact vector we are trying to
approximate. The KKT system associated with the constrained least squares problem is
\begin{equation}
\label{eq:rmkkt}
\bmat{X^TX & e \\ e^T & 0}\bmat{a\\ \lambda}=\bmat{X^Tx\\ 1}.
\end{equation}
But notice that each element of $X^TX$ is an inner product between two time histories, and this can be interpreted as
approximating the covariance function with an empirical covariance, i.e.
\begin{equation}
x_k^Tx_j\;=\;\sum_{i=1}^m x(t_i,s_k)x(t_i,s_j) \;\approx\; (m-1)\,\cov(s_k,s_j).
\end{equation}
Similarly for the point $s$,
\begin{equation}
x_k^Tx\;=\;\sum_{i=1}^m x(t_i,s_k)x(t_i,s) \;\approx\; (m-1)\,\cov(s_k,s).
\end{equation}
Then,
\begin{equation}
\frac{1}{m-1}X^TX\approx A,\qquad \frac{1}{m-1}X^Tx \approx b,
\end{equation}
where $A$ and $b$ are from \eqref{eq:krigingkkt}.We can multiply the top equations of \eqref{eq:rmkkt} by $1/(m-1)$ to
transform it to an empirical form of \eqref{eq:krigingkkt}. Loosely speaking, the residual minimizing method for
computing the coefficients $a$ is a transformed version of the least squares problem \eqref{eq:rawlsi}, where the
transformation comes from the underlying model equations. In this sense, we can think of the residual minimizing 
method as a version of kriging where the covariance information comes from sampling the underlying dynamical
model.

\section{Analysis}
\label{sec:analysis}

We begin this section with a summary its results. We first examine the quality of a linear approximation of the process
and derive a lower bound for the average error over the input parameter space; such analysis is related to the error
bounds found in POD-based model reduction~\cite{Antoulas01,Roms08}. We then show that the residual minimizing 
model interpolates the problem data; in fact, even the initial guess proposed in Section \ref{sec:initguess} has an
interpolatory property in most cases. We then show that the coefficients inherit the type of input
parameter dependence from the underlying dynamical model. In particular, if the dynamical system forcing function
depends continuously on the input parameters, then so do the coefficients; the interpolatory property and continuity of
the coefficients are appealing aspects of the scheme. Next we prove that each added basis improves the approximation
over the entire parameter space. Finally, the last theorem of the section relates the minimum singular value of the data
matrix in the Newton step to the equation residual; this result implies that to achieve an approximation with a small
residual one must solve an ill-conditioned least squares problem.

\subsection{An Error Bound}
Before addressing the characteristics of the interpolant, we may ask how well a linear combination of basis vectors 
can approximate some parameterized vector. The following theorem gives a lower bound on the
average error in the best $n$-term linear approximation in the Euclidean norm. This bound is valid for any linear
approximation of a parameterized vector -- not necessarily the time history of a parameterized dynamical system -- and
thus applies to any linear method, including those mentioned in the introduction. The lower bound is useful for
determining a best approximation. If our approximation behaves like the lower bound as $n$ increases, then we can
claim that it behaves like the best approximation. 

\begin{theorem}
\label{thm:approx}
Let $x=x(s)$ be a parameterized $p$-vector, and let $X$ be a full rank matrix of size $p\times n$ with $n<p$. Define the
symmetric, positive semi-definite matrix
\begin{equation}
C \;=\; \int_\sS xx^T \,ds
\end{equation}
and let $C=U\Theta U^T$ be its eigenvalue decomposition, where $\theta_1\geq\cdots\geq\theta_p$ are the ordered
eigenvalues. Then 
\begin{equation}
\label{eq:lowerbound}
\int_\sS \minn{a}\;\|Xa-x\|^2 \,ds \;\geq\;\sqrt{\sum_{k=n+1}^p \theta_k^2}.
\end{equation}
\end{theorem}

\begin{proof}
For a fixed $s$ with $x=x(s)$, we examine the least squares problem
\begin{equation}
\mini{a}\;\|Xa-x\|.
\end{equation}
Using the normal equations, we have
\begin{equation}
a=(X^TX)^{-1}X^Tx.
\end{equation}
Then the minimum residual is given by
\begin{equation}
\minn{a} \;\|Xa-x\| \;=\; \|(X(X^TX)^{-1}X^T-I)x\| \;=\; \|Bx\|,
\end{equation}
where $B=X(X^TX)^{-1}X^T-I$. Denote the Frobenius norm by $\|\cdot\|_F$. Taking the average of the minimum norm
squared, we have
\begin{align}
\int_\sS \|Bx\|^2 \,ds &=  \int_\sS x^TB^TBx \,ds\\
&= \int_\sS \|Bxx^TB^T\|_F \,ds \\
&\geq \left\|B\,\left(\int_\sS xx^T\,ds \right)\,B^T\right\|_F\label{eq:jensen}\\
&= \|BCB^T\|_F,
\end{align}
where \eqref{eq:jensen} comes from Jensen's inequality. Suppose that $X$ contains the first $n$ columns of $U$,
i.e.~partition
\begin{equation}
U=\bmat{X& Y}.
\end{equation}
Similarly partition the associated eigenvalues
\begin{equation}
\Theta=\bmat{\Theta_1 & \\ & \Theta_2}.
\end{equation}
Then since $X$ and $Y$ are orthogonal,
\begin{align*}
B &= X(X^TX)^{-1}X^T-I\\
&= XX^T-I\\
&= -YY^T.
\end{align*} 
In this case
\begin{align*}
\|BCB^T\|_F &= \|YY^TCYY^T\|_F\\
&= \|Y\Theta_2Y^T\|_F\\
&= \sqrt{\sum_{k=n+1}^p \theta_k^2}
\end{align*}
Therefore, for a given $X$ (not necessarily the eigenvectors),
\begin{equation}
\int_\sS \minn{a} \;\|Xa-x\|^2 \,ds \;\geq\; \|BCB^T\|_F \;\geq\; \sqrt{\sum_{k=p+1}^n \theta_k^2},
\end{equation}
as required.
\end{proof}

In words, Theorem \ref{thm:approx} states that the average optimal linear approximation error is at least as large as
the norm of the neglected eigenvalues of covariance-like matrix $C$. As an aside, we note that if $p=n$, then the best
approximation error is zero, since $X$ is an invertible matrix. 

The result in Theorem \ref{thm:approx} is similar in spirit to the approximation properties of POD-Galerkin based model
reduction techniques~\cite{Antoulas01} and the best approximation results for Karhunen-Loeve type
decompositions~\cite{Loeve78}. Given a matrix of snapshots $X$ whose $j$th column is $x(s_j)$, one can approximate the
matrix $C\approx n^{-1}XX^T$; the error bounds for POD-based reduced order models are typically given in terms of the
singular values of $X$. We can approximate this lower bound for a given problem and compare it to the error for the
residual minimizing reduced model; we will see one numerical example that the approximation error behaves like the
lower bound.

\subsection{Interpolation and Continuity}

A few important properties are immediate from the construction of the residual minimizing model.
Existence follows from existence of a minimizer for the nonlinear least squares problem \eqref{eq:nlls}. Also, by
construction, the coefficients $a$ provide the optimal approximation in the space spanned by the bases, where
optimality is with respect to the surrogate error measure given by the objective function of \eqref{eq:nlls} --
i.e., the residual -- under the constraint that $e^Ta=1$. As in other residual minimizing schemes, the minimum
value of the objective function provides an a posteriori error measure for the approximation. The next few theorems
expose additional interesting properties of the residual minimizing approximation.

\begin{theorem}
\label{thm:interp}
The residual minimizing approximation $\tilde{x}(t,s)$ interpolates the basis elements, i.e.,
\begin{equation}
\tilde{x}(t_i,s_j) = x_j(t_i), \qquad i=1,\dots,m, \quad j=1,\dots,n.
\end{equation}  
\end{theorem}

\begin{proof}
Let $a$ be a vector such that $e^Ta=1$ and 
\begin{equation}
X_ia=x_j(t_i), \quad F_ia=f(x_j(t_i),t_i,s_j),\quad i=1,\dots,m. 
\end{equation}
Such a vector exists; the $n$-vector of zeros with 1 in the $j$th entry satisfies these conditions. 
For $\rho(a)=\rho(a,s_j)$ from \eqref{eq:nlls},
\begin{align*}
\rho(a) &= \sum_{i=1}^m w_i^2 \|F_ia - f(X_ia,t_i,s_j)\|^2 \\
&= \sum_{i=1}^m w_i^2 \|f(x_j(t_i),t_i,s_j) - f(x_j(t_i),t_i,s_j)\|^2 \;=\; 0
\end{align*}
which is a minimum. Then
\begin{equation}
\tilde{x}(t_i,s_j) \;=\; X_ia \;=\; x_j(t_i), \quad i=1,\dots,m.
\end{equation}
Since $j$ was arbitrary, this completes the proof.
\end{proof}

Theorem \ref{thm:interp} justifies the label \emph{interpolant} for the approximation $\tilde{x}(t,s)$. For many cases,
this property extends to the initial guess described in Section \ref{sec:initguess}. 

\begin{theorem}
\label{thm:initinterp}
Let $R_{-1}=R_{-1}(s)$ be defined as in \eqref{eq:rzero}. Assume the matrix $[R_{-1}^T,e]^T$ has full column rank for
all $s=s_j$. Then the initial guess, denoted by $\tilde{x}_0(t,s)$, interpolates the basis elements, i.e.,
\begin{equation}
\tilde{x}_0(t_i,s_j) = x_j(t_i), \qquad i=1,\dots,m, \quad j=1,\dots,n.
\end{equation}  
\end{theorem}

\begin{proof}
The full rank assumption on $[R_{-1}^T,e]^T$ at each $s=s_j$ implies that the constrained least squares problem
\begin{equation}
\begin{array}{ll}
\mini{a} & \|R_{-1}a\| \\
\st       & e^Ta = 1
\end{array}
\end{equation}
has a unique solution. Let $a$ be a vector of zeros with 1 in the $j$th element. Then $e^Ta=1$ and
\begin{align*}
\|R_{-1}a\| &= \|Fa - Ga\| \\
&= \left\|
\bmat{w_1 f(x_j(t_1),t_1,s_j) \\ \vdots \\ w_m f(x_j(t_m),t_m,s_j)}
- \bmat{w_1 f(x_j(t_1),t_1,s_j) \\ \vdots \\ w_m f(x_j(t_m),t_m,s_j)}
\right\| \;=\; 0,
\end{align*}
which is a minimum. By the uniqueness,
\begin{equation}
\tilde{x}_0(t_i,s_j) \;=\; X_ia \;=\; x_j(t_i), \quad i=1,\dots,m.
\end{equation}
Since $j$ was arbitrary, this completes the proof.
\end{proof}

The proof required the full rank assumption on $[R_{-1}^T,e]^T$ at each $s=s_j$. We expect this to be the case in
practice with a properly selected bases $x_j$. If we extend this assumption to all $s\in\sS$, then we make a statement
about the how the coefficients of the interpolant behave over the parameter space. In the next theorem, we
show that the coefficients of the residual minimizing interpolant inherit the behavior of the model terms with respect
to the input parameters.

\begin{theorem}
\label{thm:lsanalyticbases} 
Let $R_k=R_k(s)$ be defined as in \eqref{eq:defRk}, and assume the matrix $[R_k^T,e]^T$
has full column rank for all $s\in\sS$. If the function $f(x,t,s)$ and its Jacobian $\nabla_xf$ depend continuously on
$s$, then the coefficients $a(s)$ computed with a finite number of Newton iterations are also continuous with respect to $s$. 
\end{theorem}

\begin{proof}
We can write the KKT system associated with the constrained least squares problem \eqref{eq:newton} as
\begin{equation}
\label{eq:kkt}
\bmat{R_k^T R_k & e \\ e^T & 0}\bmat{a\\ \lambda}=\bmat{\mathbf{0} \\ 1},
\end{equation}
where $\lambda=\lambda(s)$ is the Lagrange multiplier. The full rank assumption on $[R_k^T,e]^T$ implies that the KKT
matrix is invertible for all $s\in\sS$, and we can write
\begin{equation}
\bmat{a\\ \lambda}=\bmat{R_k^TR_k & e \\ e^T & 0}^{-1}\bmat{\mathbf{0} \\ 1}.
\end{equation}
Since the elements of $R_k$ are continuous in $s$, so are the elements of the inverse of the KKT system, which implies
that $a(s)$ is continuous in $s$, as required.
\end{proof}

\subsection{Convergence and Conditioning}

The convergence analysis we present is somewhat nonstandard. As opposed to computing an a priori rate of convergence,
we show in the next theorem that the minimum residual decreases monotonically with each added basis. 

\begin{theorem}
\label{thm:addpoint}
Let
\begin{equation}
\label{eq:nbases}
X^{(n)} = \bmat{X_1\\ \vdots \\ X_m} 
\end{equation}
be a given basis with $n$ columns, and let $a_n^\ast$ be the minimizer of \eqref{eq:newton} with associated minimum
function value $\rho^\ast_n=\rho(a_n^\ast,s)$. Let $s_{n+1}\in\sS$ and 
\begin{equation}
x_{n+1} = \bmat{x(t_1,s_{n+1})\\ \vdots \\ x(t_m,s_{n+1})}.
\end{equation}
Define the updated basis 
\begin{equation}
X^{(n+1)}=\bmat{X^{(n)}& x_{n+1}}.
\end{equation}
Let $a_{n+1}^\ast$ be the minimizer of \eqref{eq:newton} with the basis $X^{(n+1)}$, and let
$\rho^\ast_{n+1}=\rho(a_{n+1}^\ast,s)$ be the minimum function value. Then there is an $0\leq \alpha\leq
1$ such that
\begin{equation}
\rho^\ast_{n+1} = \alpha\,\rho^\ast_n,
\end{equation}
for all $s\in\sS$. 
\end{theorem}

\begin{proof}
Define 
\begin{equation}
a=\bmat{a_n^\ast\\ 0},
\end{equation}
and note that $e^Ta=1$. Then
\begin{equation}
\rho^\ast_{n+1}\;\leq\;\rho(a,s)\;=\;\rho^\ast_n.
\end{equation}
Next let $s=s_{n+1}$, and let $a$ be a $n+1$-vector of zeros with a 1 in the last entry. Then 
\begin{equation}
\rho^\ast_{n+1}\;=\;\rho(a,s_{n+1})\;=\;0,
\end{equation}
as required.
\end{proof}

In words, the minimum residual decreases monotonically as bases are added to the approximation.
Since $s$ is arbitrary, Theorem \ref{thm:addpoint} says that each basis added to the approximation reduces the
residual norm \emph{globally} over the parameter space; in the worst case, it does no harm, and in the best case it
achieves the true solution. This result is similar to Theorem 2.2 from~\cite{Bui08}. 

In the final theorem of this section, we show that achieving a small residual requires the solution of an
ill-conditioned constrained least squares problem for the Newton step. To do this, we first reshape the constrained
least squares problem using the standard null space method~\cite{Bjorck96}, and then we apply a result
from~\cite{VanHuffel91} relating the minimum norm of the residual to the minimum singular value of the data matrix.

\begin{theorem}
\label{thm:illcon}
Let $R_k$ be the matrix from the constrained least squares problem for the Newton step \eqref{eq:newton}, and define
$\smin(R_k)$ to be its minimum singular value. Then
\begin{equation}
\smin(R_k) \;\leq\;
\sqrt{n}\left(
\|\sJ\|\|a_{k+1}-a_k\| + \sqrt{\rho(a_k)}
\right).
\end{equation}
\end{theorem}

\begin{proof}
Let $R=R_k$ be defined as in \eqref{eq:defRk}. To solve the constrained least squares problem \eqref{eq:newton} via the
nullspace method, we take a QR factorization of the constraint vector
\begin{equation}
Q^T e = \sqrt{n} e_1,
\end{equation}
where $e_1$ is an $n$-vector of zeros with a one in the first entry. Partition $Q = [q_1,Q_2]$, where $q_1$ is the first
column of $Q$. The $n\times(n-1)$ matrix $Q_2$ is a basis for the null space of the constraint. It can be shown that
$q_1=n^{-1/2}e$. Using this transformation, the constrained least squares problem becomes the unconstrained least
squares problem
\begin{equation}
\label{eq:unconstrained}
\begin{array}{ll}
\mini{v} & \|RQ_2v + \frac{1}{n}Re\|.
\end{array}
\end{equation}
Then the Newton update is given by
\begin{equation}
a_{k+1} = \frac{1}{n}e + Q_2v^\ast,
\end{equation}
where $v^\ast$ is the minimizer of \eqref{eq:unconstrained}. For a general overdetermined least squares problem
$\minn{u}\;\|Au-b\|$, Van Huffel~\cite{VanHuffel91} shows that for $\gamma>0$,
\begin{equation}
\frac{\smin([b\gamma,A])}{\gamma}\leq \|b-Au^\ast\|,
\end{equation}
where $u^\ast$ is the minimizer; we apply this result to \eqref{eq:unconstrained}. Taking $\gamma=\sqrt{n}$, we have
\begin{equation}
\left[\frac{1}{\sqrt{n}}Re,RQ_2\right] = RQ.
\end{equation}
Since $Q$ is orthogonal, $\smin(RQ)=\smin(R)$. Then
\begin{equation}
\smin(R) \;\leq\; \sqrt{n} \left\|RQ_2v^\ast + \frac{1}{n}Re\right\| \;=\; \sqrt{n} \|Ra_{k+1}\|.
\end{equation}
Combining this result with the bound on the residual \eqref{eq:residbound} achieves the desired result.
\end{proof}

To reiterate, near the solution of the nonlinear least squares problem \eqref{eq:nlls}, we expect $\|\sJ\|$ to be
bounded and the difference $\|a_{k+1}-a_k\|$ to be small so that the first term in the bound becomes negligible.
Therefore, a solution with a small residual norm $\rho(a)$ will require the solution of a constrained least squares
problem with a small minimum singular value, which implies a large condition number. To combat this, we propose some
heuristics for dealing with the ill-conditioning of the problem in the following section.

\section{Computational Heuristics}
\label{sec:heuristics}

In this section, we propose heuristics for (1) reducing the cost of constructing the interpolant, (2) mitigating
the ill-conditioning of the least squares problems, and (3) adding new bases to the approximation.

\subsection{Cost Reduction}
\label{sec:cred}

The reader may have noticed that, if we measure the cost of the residual minimizing scheme in terms of number of
evaluations of $f$ from \eqref{eq:model} (assuming we use the finite difference Jacobian described in Section
\ref{sec:fdjacob}), then computing the approximation can be more expensive than evaluating the true model. Specifically,
an explicit one-step time stepping scheme evaluated at the time discretization $t_1,\dots,t_m$ requires
$m$ evaluations of $f$. However, merely constructing the matrix $G$ in \eqref{eq:rzero} at the same time
discretization to compute the initial guess requires $mn$ evaluations of $f$ for $n$ bases. And if the elements of the
matrix $F$ were not stored during the runs, computing them requires another $mn$ evaluations of $f$. On top of that,
each Newton step requires yet another $mn$ evaluations. This simple cost analysis would seem to discourage us from
comparing the interpolation method to other methods for model reduction.

The interpolation method becomes an appropriate method for model reduction when one is interested in a small subset of
the time domain. For example, if a quantity of interest is computed as a function of the state at some final time, then
the interpolation method can be used to approximate the state at the final time at a new parameter value without
computing the full history. This is particularly useful for dynamical systems with rapidly varying initial transients
that require small time steps. With the interpolation method, one need not resolve the initial transients with a
small time step to approximate the state at the final time. Instead, the method takes advantage of regularity in the
parameter space to construct the interpolating approximation. 

Alternatively, if one were interested in a minimal set of points in time for coarse approximation of the history, these
could be determined with a method similar to the discrete empirical interpolation method~\cite{Chaturantabut10} in the
time domain. With this set of discretization points, one could approximate the time history at a new parameter value
without the full model solver. We do not pursue this idea further in this work, but we believe it holds promise for
approximating time-averaged quantities of interest.

This sort of reduction is applicable to the number of points $m$ in the time domain. In the parameter domain, the number
of points $n$ corresponds to the number of full model solutions used in the affine combination \eqref{eq:linmodel}.
Typically, we think that $n$ is small due to the cost of the full model solution. However, in cases where a state
vector may be represented by an even smaller set of basis vectors, we can borrow an idea from moving (or windowed)
least squares~\cite{Wendland05} to reduce the number of columns in the solves for the Newton steps \eqref{eq:newton}. 

Choose an integer $M<n$. For a parameter point $s$, find the $M$ points in the set of $\{s_j\}$ that are nearest to $s$.
Then use only the $x_j$ corresponding to nearby $s_j$ to construct the least squares approximations. In the numerical
examples, we demonstrate the savings generated by this heuristic.

\subsection{Alleviating Ill-Conditioning}
\label{sec:illcon}

In Theorem \ref{thm:illcon}, we showed that achieving a small residual required the solution of an ill-conditioned least
squares problem for the Newton step. Here we offer a heuristic for alleviating the effects of the ill-conditioning.
The heart of the residual minimizing scheme is the constrained linear least squares problem \eqref{eq:newton} which we
rewrite with a general $m\times n$ matrix $R$ as
\begin{equation}
\label{eq:cls}
\begin{array}{ll}
\mini{a} & \| Ra \| \\
\st       & e^Ta= 1.
\end{array}
\end{equation}
For now we assume that $R$ is full rank, although it may have a very large condition number. The particular form of the
least squares problem (the single linear constraint $e$ and the zero right hand side) will permit us to use some novel
approaches for dealing with the ill-conditioning.

To analyze this problem, we first derive a few useful expressions. Let $\lambda$ be the Lagrange multiplier
associated with the constraint. The minimizer $[\astar^T,-\lstar]^T$ satisfies the KKT conditions
\begin{equation}
\bmat{R^TR & e\\ e^T & 0}\bmat{\astar\\ -\lstar} = \bmat{0\\ 1}.
\end{equation}
We can use a Schur complement (or block elimination) method to derive expressions for the solution: 
\begin{equation}
\lstar = \frac{1}{e^T(R^TR)^{-1}e},\qquad
\astar = \lstar(R^TR)^{-1}e.
\end{equation}
Also note that the first KKT equation gives
\begin{equation}
R^TR\astar = \lstar e.
\end{equation}
Premultiplying this by $\astar^T$ , we get
\begin{equation}
\astar^TR^TR\astar = \lstar \underbrace{\astar^Te}_{=\,1},
\end{equation}
or equivalently
\begin{equation}
\label{eq:normlagrange}
\|R\astar\|^2 = \lstar.
\end{equation}
In other words, the value of the optimal Lagrange multiplier is the squared norm of the residual. We can use this fact
to improve the conditioning of computing the approximation while maintaining a small residual.

Since we wish to minimize $\|Ra\|$, it
is natural to look for a linear combination of the right singular vectors of $R$ associated with the smallest singular
values. We first compute the thin singular value decomposition (SVD) of $R$ as
\begin{equation}
R=U\Sigma V^T \qquad \mathrm{diag}(\Sigma)=[\sigma_1,\dots,\sigma_n],
\end{equation}
and rotate \eqref{eq:cls} by the right singular vectors. Define
\begin{equation}
\ahat = V^Ta, \qquad d = V^Te.
\end{equation}
Then \eqref{eq:cls} becomes
\begin{equation}
\begin{array}{ll}
\mini{a} & \| \Sigma \ahat \| \\
\st       & d^T\ahat= 1,
\end{array}
\end{equation}
with associated KKT conditions
\begin{equation}
\bmat{\Sigma^2 & d\\ d^T & 0}\bmat{\ahatstar\\ -\lstar} = \bmat{0\\ 1}.
\end{equation}
Applying block elimination to solve this system involves inverting $\Sigma^2$ and multiplying by $d$ -- an operation
with a condition number of $(\sigma_1/\sigma_n)^2$. We can improve the conditioning by solving a truncated version of
the problem. Let $k<n$ be a truncation and partition
\begin{equation}
\Sigma = \bmat{\Sigma_1 & \\ & \Sigma_2}
\qquad V=\bmat{V_1 & V_2} 
\qquad d=\bmat{d_1\\ d_2} 
\qquad \ahat=\bmat{\ahat_1\\\ahat_2}
\end{equation}
according to $k$. Since we want to minimize the residual $\|\Sigma\ahat\|$ and the largest values of $\Sigma$ appear at
the top, we set $\ahatonestar=0$ and solve
\begin{equation}
\bmat{\Sigma_2^2 & d_2\\ d_2^T & 0}\bmat{\ahattwostart\\ -\lstart} = \bmat{0\\ 1},
\end{equation}
where the superscript $(t)$ is for \emph{truncated}; the condition number of inverting and multiplying by
$\Sigma_2^2$ is $(\sigma_{k+1}/\sigma_n)^2$. Then the linear combination of the right singular vectors associated with
the smallest singular values is
\begin{equation}
\astar^{(t)} = V_2^T\ahattwostart.
\end{equation}
We can measure what was lost in the truncation in terms of the minimum residual by looking at the difference between the
Lagrange multipliers
\begin{equation}
\label{eq:residdiff}
\left| \|R\astar\|^2 - \|R\astar^{(t)}\|^2 \right| \;=\; |\lstar - \lstart|.
\end{equation}
And this suggests a way to choose the truncation a priori. Once the SVD of $R$ has been computed, we can trivially check the
minimum residual norm by computing the Lagrange multipliers for various truncations; see (\ref{eq:normlagrange}). We
want to truncate to reduce the condition number of the problem and avoid disastrous numerical errors. But we do not
want to truncate so much that we lose accuracy in the approximation. Therefore, we propose the following. For
$j=1,\dots,n$, compute
\begin{equation*}
\begin{array}{rll}
d_j &= [d_{n-j+1},\dots,d_n]^T &\mbox{the last $j$ elements of $d$}\\
\Sigma_j &= \mathrm{diag}(\sigma_{n-j+1},\dots,\sigma_n) &\mbox{the last $j$ singular values}\\
y_j &= \Sigma_j^{-1}d_j & \\
\lambda_j &= \frac{1}{y_j^Ty_j} &
\end{array}
\end{equation*}
The $\lambda_j$ measure the norm of the residual for a truncation that retains the last $j$ singular
values; $\lambda_n$ is the true minimum residual. Therefore, we seek the smallest $k$ such that
\begin{equation}
|\lambda_k - \lambda_n| < \tau
\end{equation}
for a given tolerance $\tau$ representing how much minimum residual we are willing to sacrifice for better conditioning,
and this determines our truncation.

\subsubsection{A Note on Rank Deficiency}

There is an interesting tension in the constrained least squares problem \eqref{eq:cls}. The ideal scenario would be to
find a vector in the null space of $R$ that also satisfies the constraint. This implies that we want $R$ to be rank
deficient for the sake of the approximation. However, if the matrix $[R^T,e]^T$ is rank deficient, then there are
infinite solutions for the constrained least squares problem~\cite{Bjorck96}. To distinguish between these two
potential types of rank deficiency, we apply the following. Suppose $R$ is rank deficient and we have partitioned
\begin{equation}
R = 
\bmat{U_1 & U_2}\bmat{\Sigma_1& \\ & 0}\bmat{V_1^T\\ V_2^T}.
\end{equation}
Next compute $d=V_2^Te$. All components of $d$ equal to zero correspond to singular vectors that are in the null space
in the constraint; we want to avoid these. Suppose
\begin{equation}
d \;=\; \bmat{d_1\\ 0} \;=\; \bmat{V_{2,1}^Te\\ V_{2,2}^Te}.
\end{equation}
Then we can choose any linear combination of the vectors $V_{2,1}$ that satisfies the constraint to achieve a zero
residual; one simple choice is to take the component-wise average of the vectors $V_{2,1}$.

If all of $d=0$, then we must apply a method for the rank deficient constrained least squares problem, such as the
null space method to transform the \eqref{eq:cls} to an unconstrained problem coupled with the pseudoinverse to compute
the minimum norm solution~\cite{Bjorck96}.

The drawback of this SVD-based approach is that each Newton step requires computing the SVD. And we
expect that most parameter and uncertainty studies will require many such evaluations. Therefore we are actively
pursuing more efficient methods for solving \eqref{eq:cls}.

\subsection{Adding Bases}
\label{sec:addbases}

When doing a parameter study on a complex engineering system, the question of where in the parameter space to compute
the solution arises frequently. In the context of this residual minimizing model interpolation, we have a natural method
and metric for answering this question. Let $\rho^\ast=\rho^\ast(s)$ be the minimum objective function value from the
nonlinear least squares problem \eqref{eq:nlls} with a given basis. Then the point $s_{n+1}\in\sS$ to next
evaluate the full model is the \emph{maximizer} of
\begin{equation}
\label{eq:nextbasis}
\begin{array}{ll}
\maxi{s\in\sS} & \rho^\ast(s),
\end{array}
\end{equation}
so that $x_{n+1}(t)=x(t,s_{n+1})$. Of course, the optimization problem \eqref{eq:nextbasis} is in general non-concave,
and derivatives with respect to the parameters $s$ are often not available. However, we are comforted by the fact
that \emph{any} point in the parameter space with a positive value for $\rho^\ast$ will yield a basis element that will
improve the approximation; see Theorem \ref{thm:addpoint}. We therefore expect
derivative-free global optimization heuristics~\cite{Kolda03} to perform sufficiently well for many applications. The
idea of maximizing the residual over the parameter space to compute the next full model solution was also proposed
in~\cite{Bui08} in the context of reduced order modeling, and they present a thorough treatment of the resulting
optimization problem -- including related greedy sampling techniques~\cite{Grepl05}.
 
There are many possible variations on \eqref{eq:nextbasis}. We have written the optimization to chose a single basis to
add. However, multiple independent optimizations could be run, and a subset of the computed optima could be added to
accelerate convergence.

\section{Numerical Examples}
\label{sec:examples}

We numerically study two models in this section: (i) a three-state nonlinear dynamical system representing a chemical
kinetics mechanism with a single input parameter controlling the stiffness of the system, and (ii) a
nonlinear, parabolic PDE modeling conductive heat transfer with a temperature dependent random field model of the
thermal conductivity.
The first problem is a toy model used to explore and confirm properties of the approximation scheme; we do not attempt
any reduction in this case. The second example represents a step toward a large scale application in need of model
reduction. All numerical experiments were performed on a dual quad-core Intel W5590 with 12GB of RAM running Ubuntu and
MATLAB 2011b. Scripts for reproducing the experiments can be found at \url{www.stanford.edu/~paulcon/rmmi.html}; the
second experiment requires the MATLAB PDE Toolbox.

\subsection{A 3-Species Kinetics Problem}

The following model from~\cite{Valorani05} contains the relevant features of a stiff chemical kinetics mechanism. Let 
\begin{equation}
x(t,s) = 
\bmat{u(t,s)\\ v(t,s)\\ w(t,s)}
\end{equation}
be the three state variables with evolution described by the function
\begin{equation}
\label{eq:kinforce}
f(x,t,s) = 
\bmat{
-5\frac{u}{s} - \frac{uv}{s} + vw + 5\frac{v^2}{s} + \frac{w}{s} - u\\
10\frac{u}{s} - \frac{uv}{s} - vw - 10\frac{v^2}{s} + \frac{w}{s} + u\\
\frac{uv}{s} - vw - \frac{w}{s} + u
},
\end{equation}
where the parameter $s$ controls the stiffness of the system; the initial state is $u_0=v_0=w_0=0.5$. We examine the
parameter range $s\in[0.005, 1.2]$, where the smaller value of $s$ corresponds to a stiffer system. Given $s$, we solve
for the evolution of the states using MATLAB's \texttt{ode45} routine from $t=0$ to $t=1$. The ODE solver has its own
adaptive time stepping method. We extract a discrete time process on 300 equally spaced points in the time interval
$[0,1]$; denote these points in time by $t_j$ where
\begin{equation}
t_j = j\Delta t, \qquad \Delta t = \frac{1}{300}.
\end{equation}
In Figure \ref{fig:kinetics}, we plot the evolution of each component for $s=0.005$ and $s=1.2$.
We can see that the stiffness of the equation yields rapidly varying initial transients for $s$ near zero,
which makes this a challenging problem.

\begin{figure}
\begin{center}
\subfloat[]{
\includegraphics[scale=0.32]{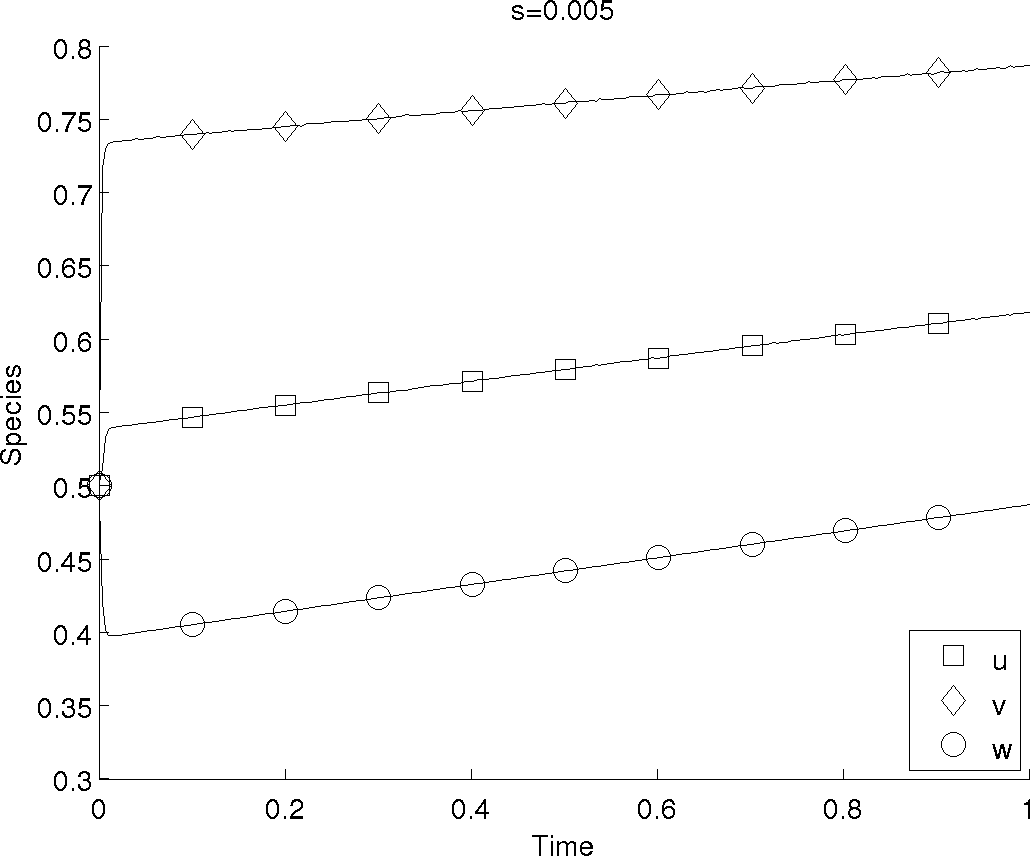}
\label{fig:stiffy}
}
\subfloat[]{
\includegraphics[scale=0.32]{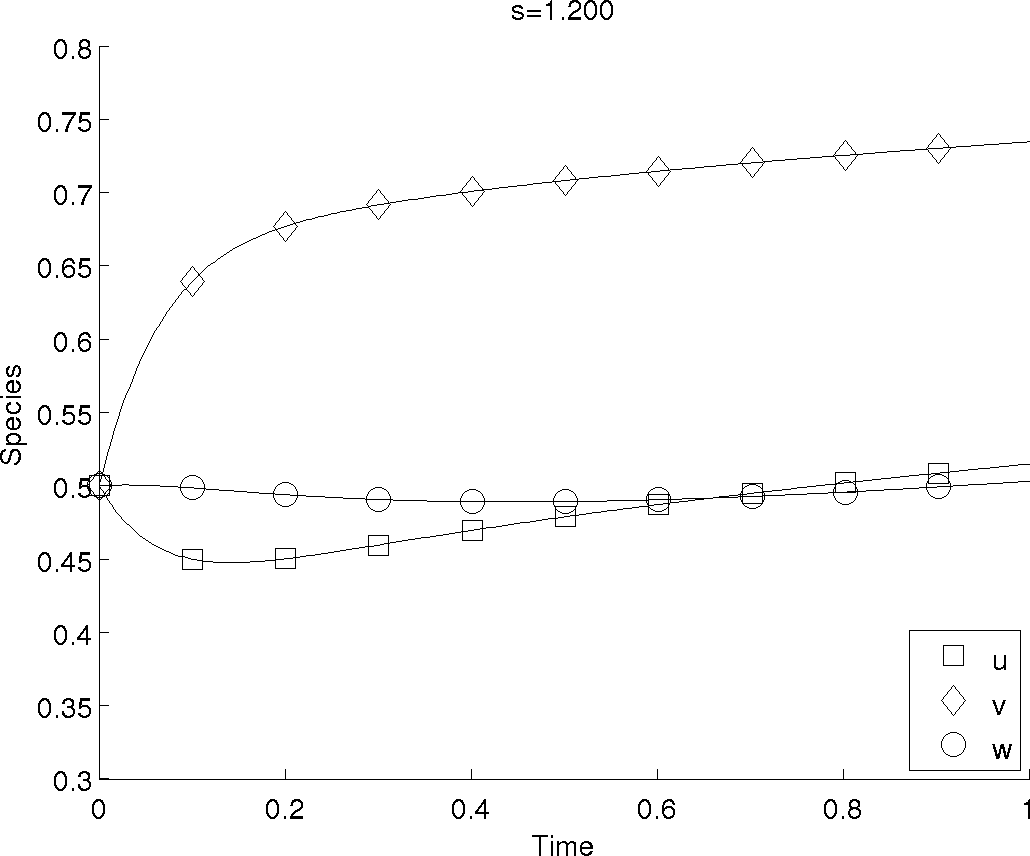}
\label{fig:nonstiffy}
}
\end{center}
\caption{The evolution of the states described by \protect\eqref{eq:kinforce} for \protect\subref{fig:stiffy} $s=0.005$
corresponding to a stiff system and \protect\subref{fig:nonstiffy} $s=1.200$ corresponding to a nonstiff system.}
\label{fig:kinetics}
\end{figure}

We approximate the average error and average minimum residual over the parameter space by using 300 equally
spaced points in the parameter range $[0.005,1.2]$; denote these points by 
\begin{equation}
s_k = 0.005 + k\Delta s,\qquad \Delta s = \frac{1.195}{300}. 
\end{equation}
The approximate average error and minimum residual are computed as
\begin{align}
\mathcal{E} &= \Delta s\,\Delta t\, \sum_k \sum_j \|\tilde{x}(t_j,s_k) - x(t_j,s_k)\|, \label{eq:EE}\\
\mathcal{R} &= \Delta s\,\Delta t\, \sum_k \sum_j \|\ddt{\tilde{x}}(t_j,s_k) - f(\tilde{x}(t_j,s_k),s_k)\|.
\label{eq:RR}
\end{align}
We examine the reduction in $\mathcal{E}$ and $\mathcal{R}$ as more bases are added according to the heuristic in
Section \ref{sec:addbases}, and we compare that with the lower bound from Theorem \ref{thm:approx}. We compute the lower
bound from Theorem \ref{thm:approx} (the right hand side of \eqref{eq:lowerbound}) using a Gauss-Legendre quadrature
rule with 9,600 points in the range of $s$, which was sufficient to obtain converged eigenvalues. 

We also use the
discretization with $s_k$ to approximate the optimization problem \eqref{eq:nextbasis} used to select additional basis
elements.
In other words, for a given basis we compute the minimum residual at each $s_k$, and we append to the basis the
true solution corresponding to the $s_k$ that yields the largest minimum residual. We begin with 2 bases -- one at each
end point of the parameter space -- and stop after 40 have been added. $\mathcal{E}$ and $\mathcal{R}$ appear in Figure
\ref{fig:convergence} compared to the computed lower bound. We see that the error behaves roughly like the lower bound.
For each evaluation of the reduced order model, we track the number of Newton iterations, and we plot the average over
the 300 $s_k$ in Figure \ref{fig:convergence}. Notice that the number of necessary Newton iterations
decreases as the number of basis functions increases.

\begin{figure}
\begin{center}
\subfloat[]{
\includegraphics[scale=0.32]{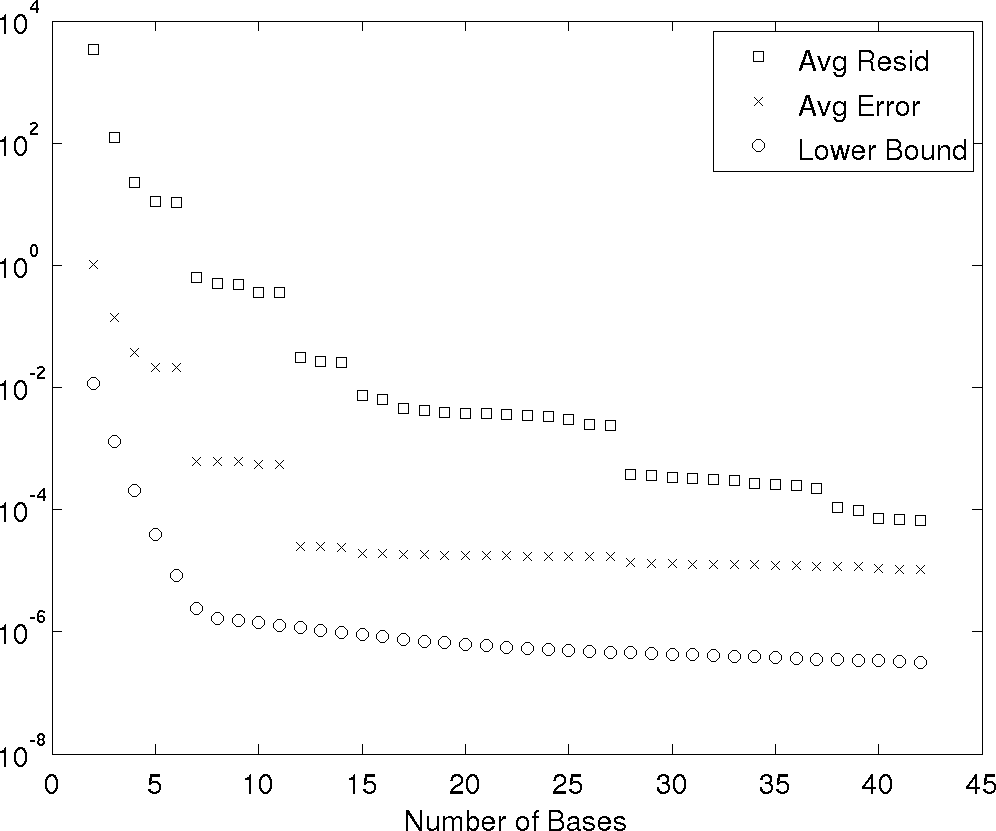}
\label{fig:conv}
}
\subfloat[]{
\includegraphics[scale=0.32]{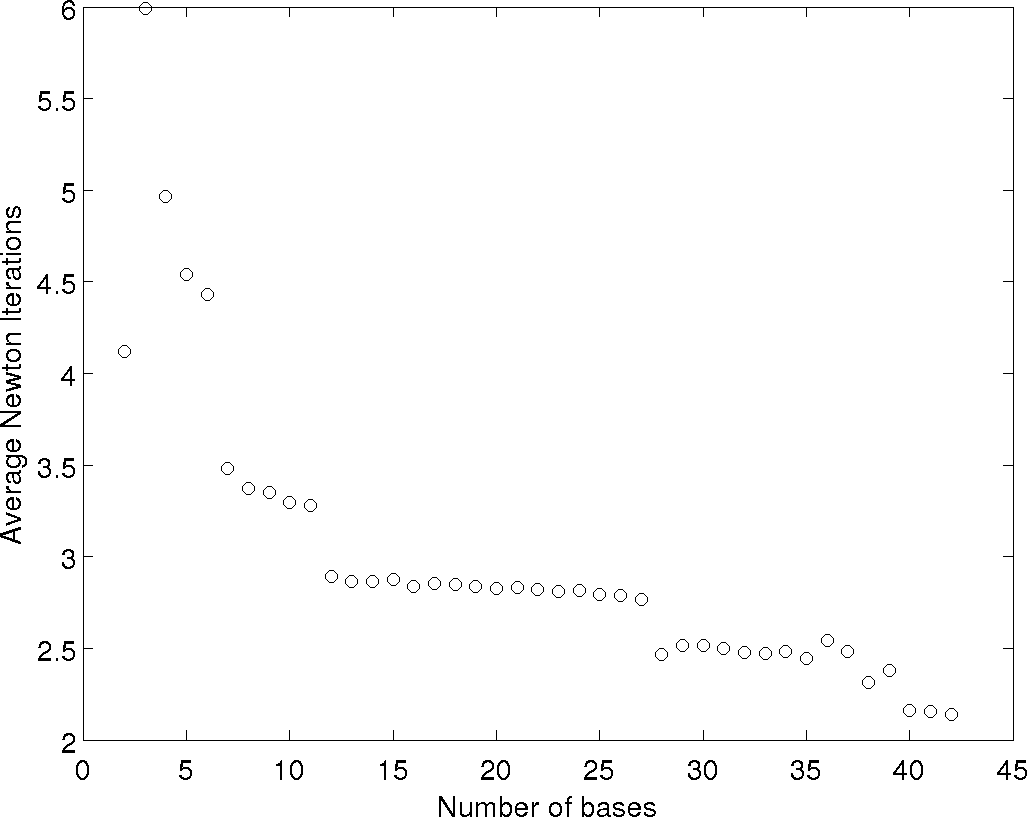}
\label{fig:newtiter}
}
\end{center}
\caption{\protect\subref{fig:conv} Reduction in average minimum residual $\mathcal{R}$ \protect\eqref{eq:RR}, average
error $\mathcal{E}$ \protect\eqref{eq:EE}, and lower bound from \protect\eqref{eq:lowerbound} as bases are added
according the heuristic in Section \protect\ref{sec:addbases}. \protect\subref{fig:newtiter} Average number of necessary
Newton iterations as basis elements are added.}
\label{fig:convergence}
\end{figure}

To confirm the results of Theorem \ref{thm:illcon}, we plot the minimum residual at each parameter point with the
maximum condition number of the matrices $R_k$ from \eqref{eq:newton} used to compute the Newton steps. In Figure
\ref{fig:illcon} we show this plot for 5, 10, 20, and 40 bases. We clearly see an inverse relationship between the
condition number and the minimum residual, which is consistent with Theorem \ref{thm:illcon}.

\begin{figure}
\begin{center}
\subfloat[]{
\includegraphics[scale=0.32]{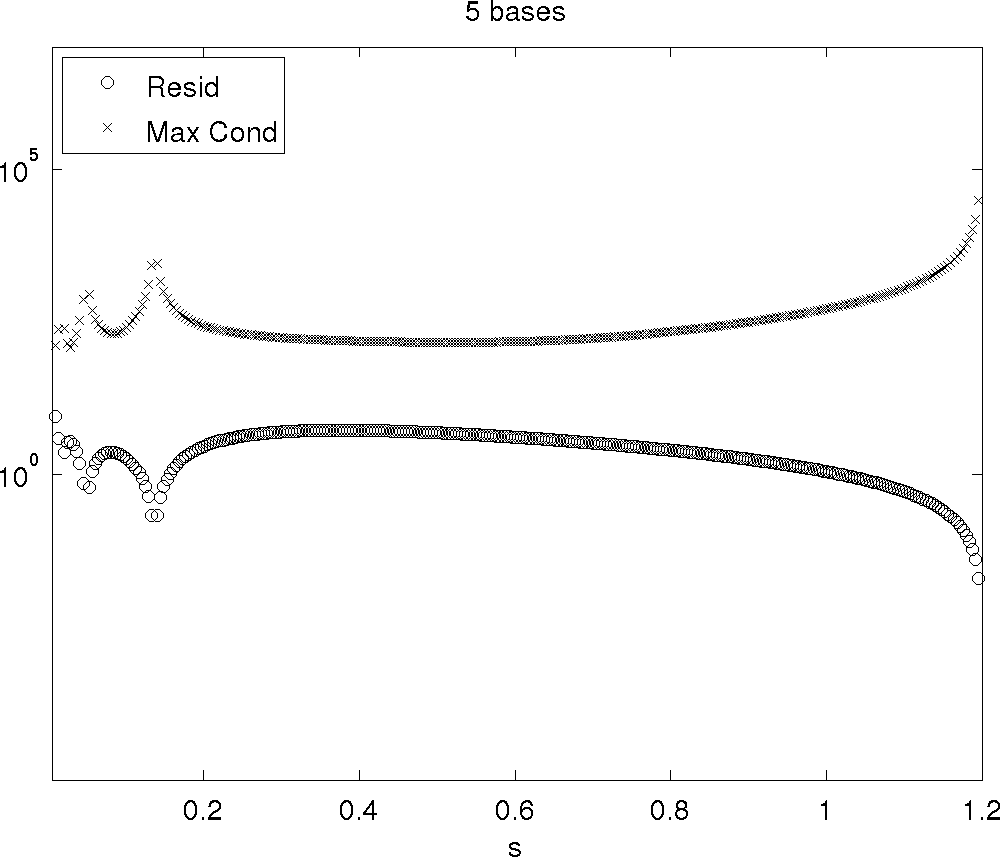}
\label{fig:res5}
}
\subfloat[]{
\includegraphics[scale=0.32]{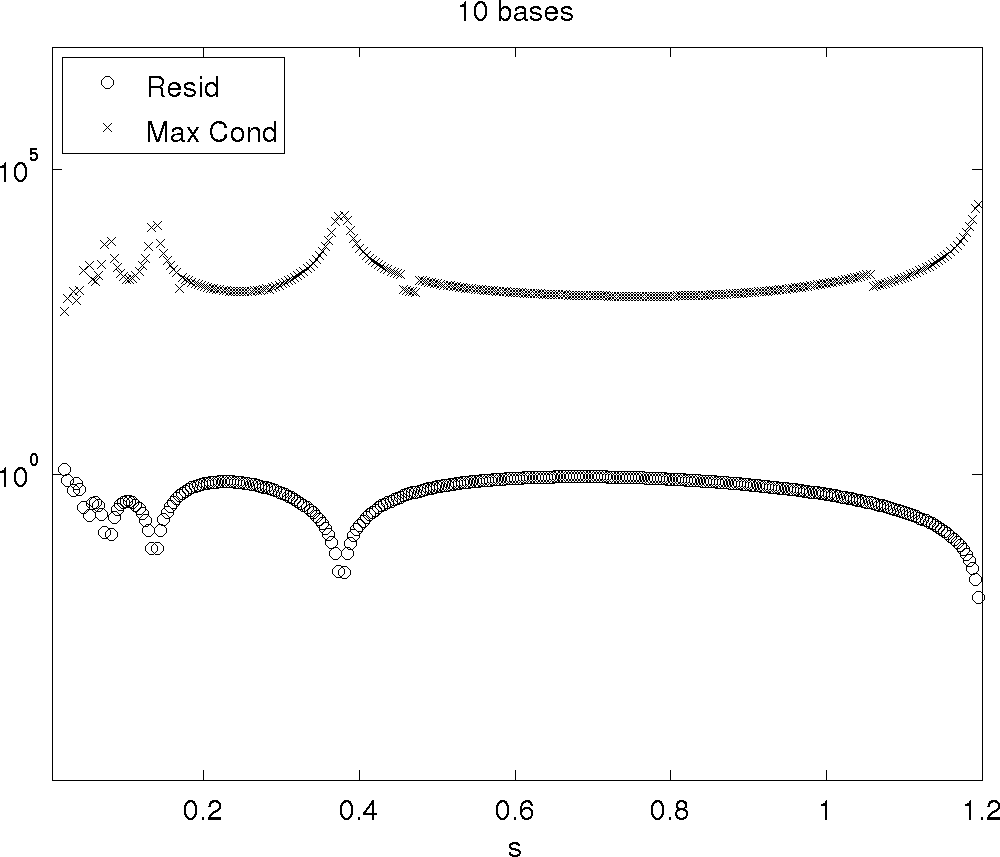}
\label{fig:res10}
}
\\
\subfloat[]{
\includegraphics[scale=0.32]{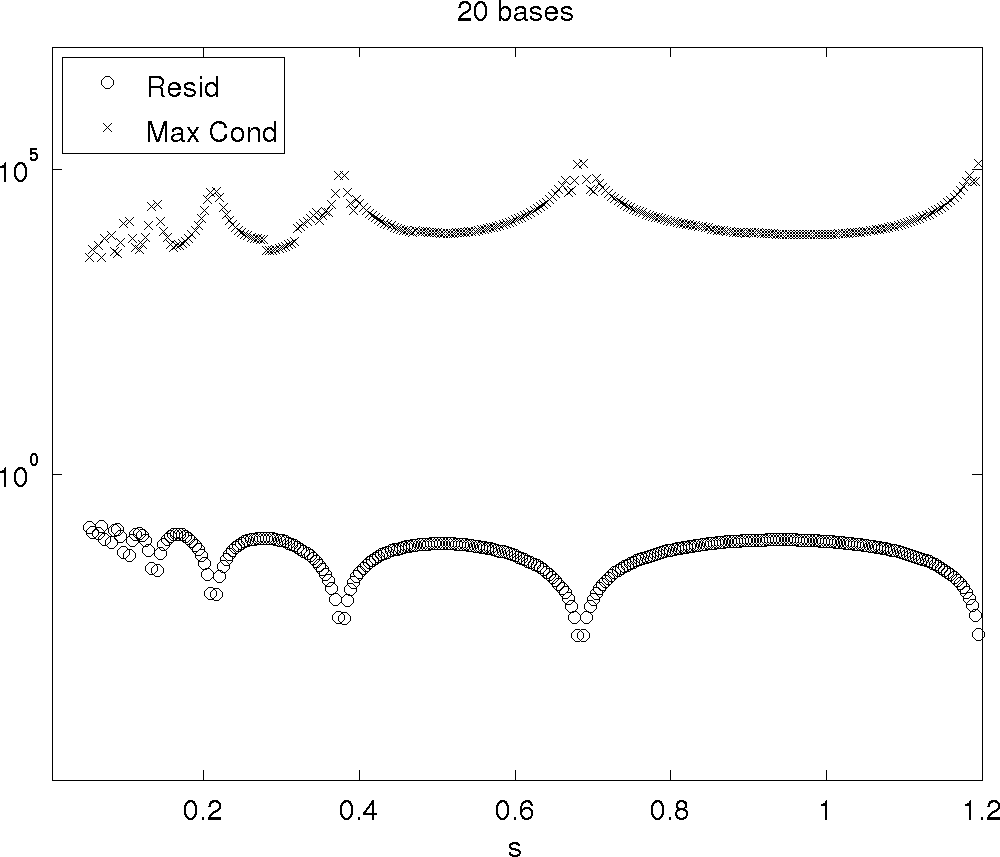}
\label{fig:res20}
}
\subfloat[]{
\includegraphics[scale=0.32]{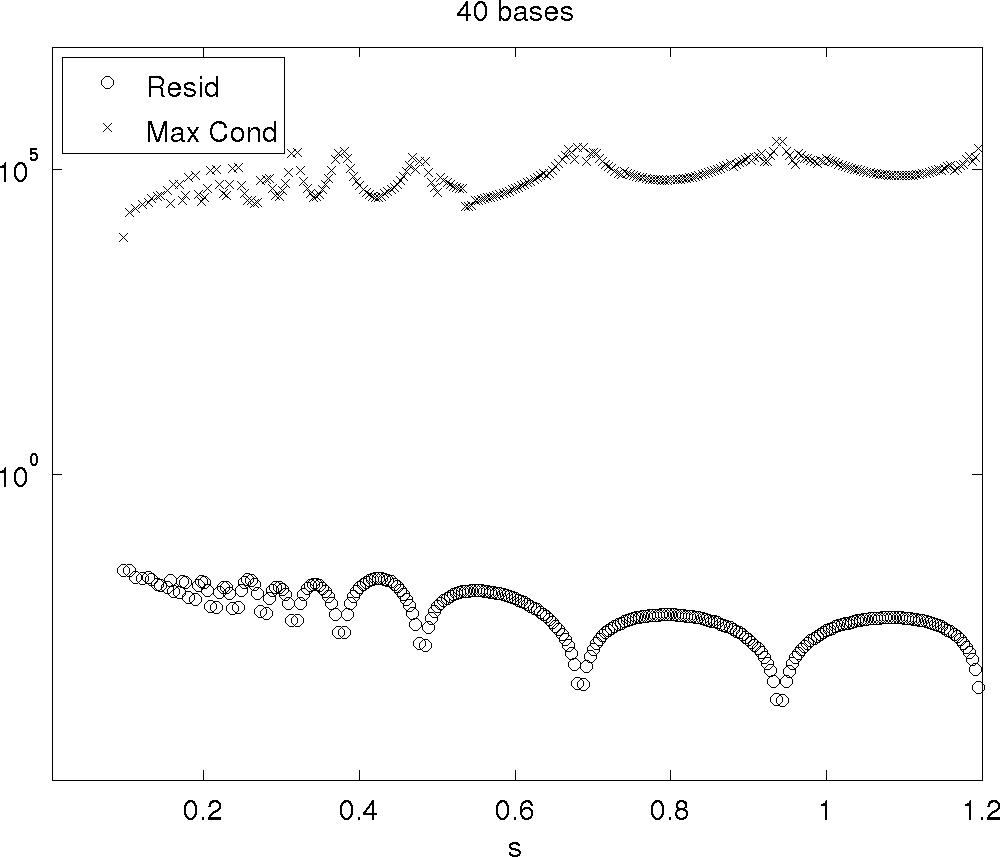}
\label{fig:res40}
}
\end{center}
\caption{Maximum condition number of $R_k$ from \protect\eqref{eq:newton} over all Newton iterations compared to minimum
residual for \protect\subref{fig:res5} 5, \protect\subref{fig:res10} 10, \protect\subref{fig:res20} 20, and
\protect\subref{fig:res40} 40 bases.}
\label{fig:illcon}
\end{figure}

\subsection{Nonlinear Transient Heat Conduction Study}

Next, we examine a two-dimensional
transient heat conduction model for a steel beam with uncertain material properties in a high temperature environment.
For a given realization of the thermal conductivity, the output of interest is the proportion of the domain whose
temperature exceeds a critical threshold after 70 seconds. We therefore need the temperature at each point in the
domain at $t=70$ seconds. Given a few simulations of the heat transfer for chosen thermal conductivities, the goal of
the model interpolation is to approximate the output of interest at other possible thermal conductivities. We are
employing the interpolation at a single point in time, which is the setting where we expect to see computational savings.

\subsubsection{Parameterized Model of Thermal Conductivity}

In~\cite{Kodur10}, Kodur and co-authors review the high temperature constitutive relationships for steel currently
used in European and American standards, as well as present the results of various experimental studies 
in the research literature. They discuss the challenges of modeling and design given the variation in the standards and
experimental data -- particularly when designing for safety in high temperature environments. We use their work to
inform a statistical model of the thermal conductivity of steel.  

To construct a statistcal model consistent with the data and codes compiled in~\cite{Kodur10}, we pose a particular
model form and choose its parameters to yield good visual agreement with the compiled data. For a given temperature $T$,
let $Y=Y(T,\omega)$ be a random variable that satisfies
\begin{equation}
Y \;=\; \bar{Y}(T) + \sigma_Y(T)\,G_Y(T,\omega).
\end{equation}
The dependence on $\omega$ signifies the random component of the model; when clear from the context, we omit
explicit dependence on $\omega$. The function $G_Y(T,\omega)$ is a standard Gaussian random field with zero mean and
two-point correlation function 
\begin{equation}
\label{eq:corr}
\mathcal{C}(T_1,T_2) = \exp\left(\frac{-(T_1-T_2)^2}{\gamma^2}\right).
\end{equation}
The parameter $\gamma$ controls the correlation between two temperatures $T_1$ and $T_2$ in the model.
The apparent smoothness of the experimental data (see Figure \ref{fig:k}) suggests a long correlation length; we choose 
$\gamma=\sqrt{500}^{\circ}\mathrm{C}$. The temperature dependent mean $\bar{Y}(T)$ is given by the the log of
the Eurocode 3 standard detailed in~\cite{Kodur10},
\begin{equation}
\label{eq:meanX}
\bar{Y}(T) = \left\{
\begin{array}{cl}
\log(-0.0333T + 54) & T<800^{\circ} \mathrm{C}\\
\log(27.30) & T\geq 800^{\circ} \mathrm{C}
\end{array}
\right.
\end{equation}
The temperature dependent function $\sigma_Y(T)$ is used to scale the variance of the model to be consistent with the
experimental data. In particular, we choose 
\begin{equation}
\sigma_Y(T) = 0.08 + 0.004\sqrt{T}. 
\end{equation}
To model the thermal conductivity $\kappa=\kappa(T,\omega)$, we take the exponential
\begin{equation}
\kappa = \exp(Y),
\end{equation}
which ensures that realizations of $\kappa$ remain positive. We employ the truncated Karhunen-Loeve
expansion~\cite{Loeve78} of the Gaussian process $G_Y$ to represent the stochasticity in $\kappa$ as a set of
independent parameters:
\begin{equation}
\label{eq:kl}
G_Y(T,\omega) \;\approx\; \sum_{i=1}^d \phi_i(T)\,\sqrt{\lambda_i}\,s_i(\omega),
\end{equation}
where $(\phi_i,\lambda_i)$ are eigenpairs of the correlation function $\mathcal{C}$ from (\ref{eq:corr}), and
$s_i=s_i(\omega)$ are a set of independent standard Gaussian random variables. The eigenfunctions $\phi_i$ from
(\ref{eq:kl}) are approximated on a uniform discretization of the temperature interval
$[0^{\circ}\mathrm{C},1250^{\circ}\mathrm{C}]$ with 600 nodes; this discretization is sufficient to capture the
correlation effects. To compute the approximate $(\phi_i,\lambda_i)$, we solve the discrete eigenvalue problem for the
symmetric, positive semidefinite correlation matrix associated with the discretization of the temperature interval. The
exponential decay of the computed eigenvalues justifies a truncation of $d=11$.

The random variables $s_i$ control the realization of the stochastic model. We can then treat them as a set of
independent input parameters for uncertainty and sensitivity analysis. To summarize, we have the following model for
thermal conductivity:
\begin{equation}
\label{eq:k}
\kappa\;=\;\kappa(T,\xi)\;=\; \exp\left[
\bar{Y}(T) + \sigma_Y(T)
	\left( 
	\sum_{i=1}^d \phi_i(T)\,\sqrt{\lambda_i}\,s_i
	\right)
\right].
\end{equation}
Note that we have been loose with the approximation step in (\ref{eq:kl}). In the end, we treat the truncated
approximation of $G_Y$ to be the true statistical model. In figure \ref{fig:k}, We plot fifty realizations of the
statistical model alongside the log of the data and standards collected in~\cite{Kodur10}.

\begin{figure}%
\centering
\subfloat[Conductivity Data]{
\includegraphics[scale=0.28]{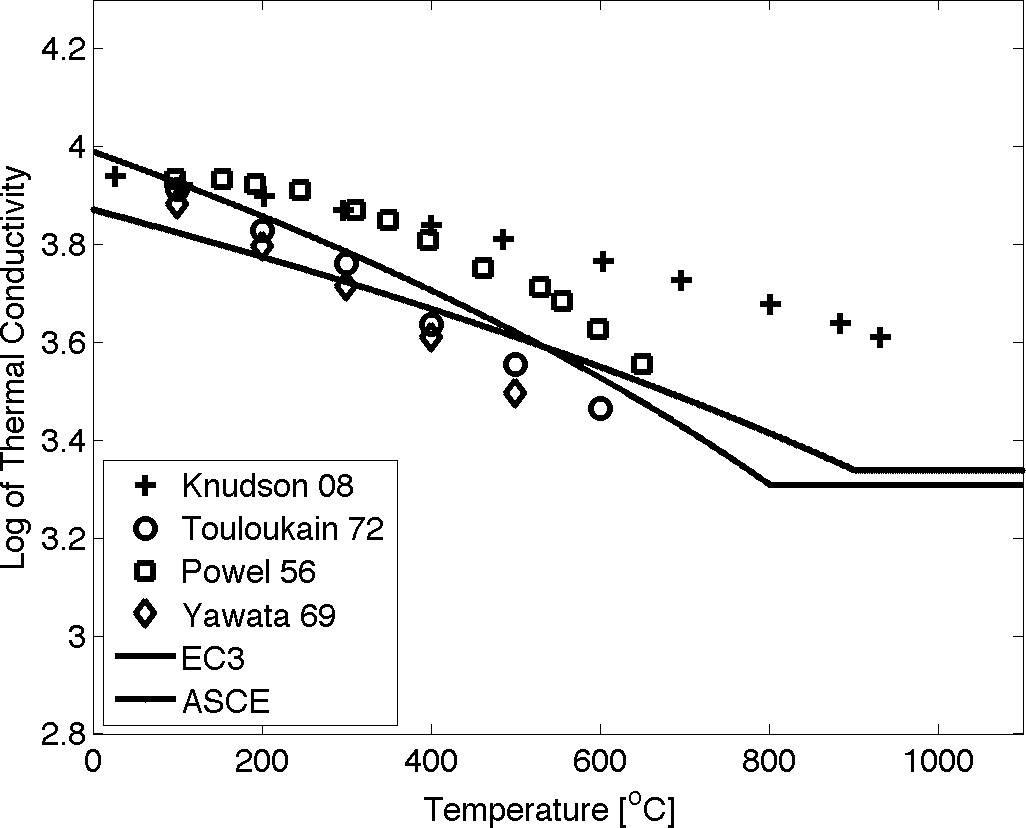}
}\qquad
\subfloat[Conductivity Model Realizations]{
\includegraphics[scale=0.28]{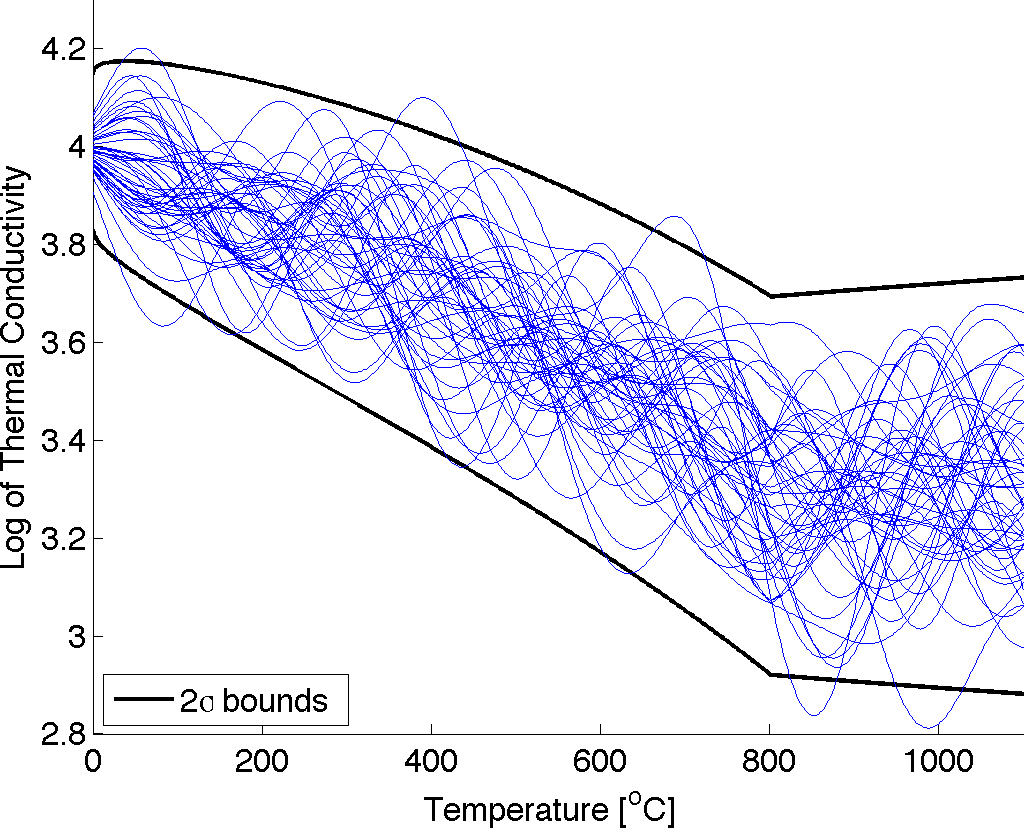}
}
\caption{Data and model realizations of temperature dependent thermal conductivity with uncertainty.}
\label{fig:k}
\end{figure}

\subsubsection{Heat Conduction Model}
Next we incorporate the statistical model for thermal conductivity into a computational simulation. The domain
$\mathcal{D}$ is a cross section of a steel beam -- shown in Figure \ref{fig:domain} with labeled boundary segments
$\Gamma_1$ and $\Gamma_2$. The boundary temperature is prescribed on $\Gamma_1$ to increase rapidly up to a maximum
value of $1100^{\circ}\mathrm{C}$; this represents rapid heating due to a fire. The boundary segment on $\Gamma_2$ is
assigned a zero heat flux condition. 

\begin{figure}
 \centering
	\subfloat[Domain]{
	\includegraphics[scale=0.28]{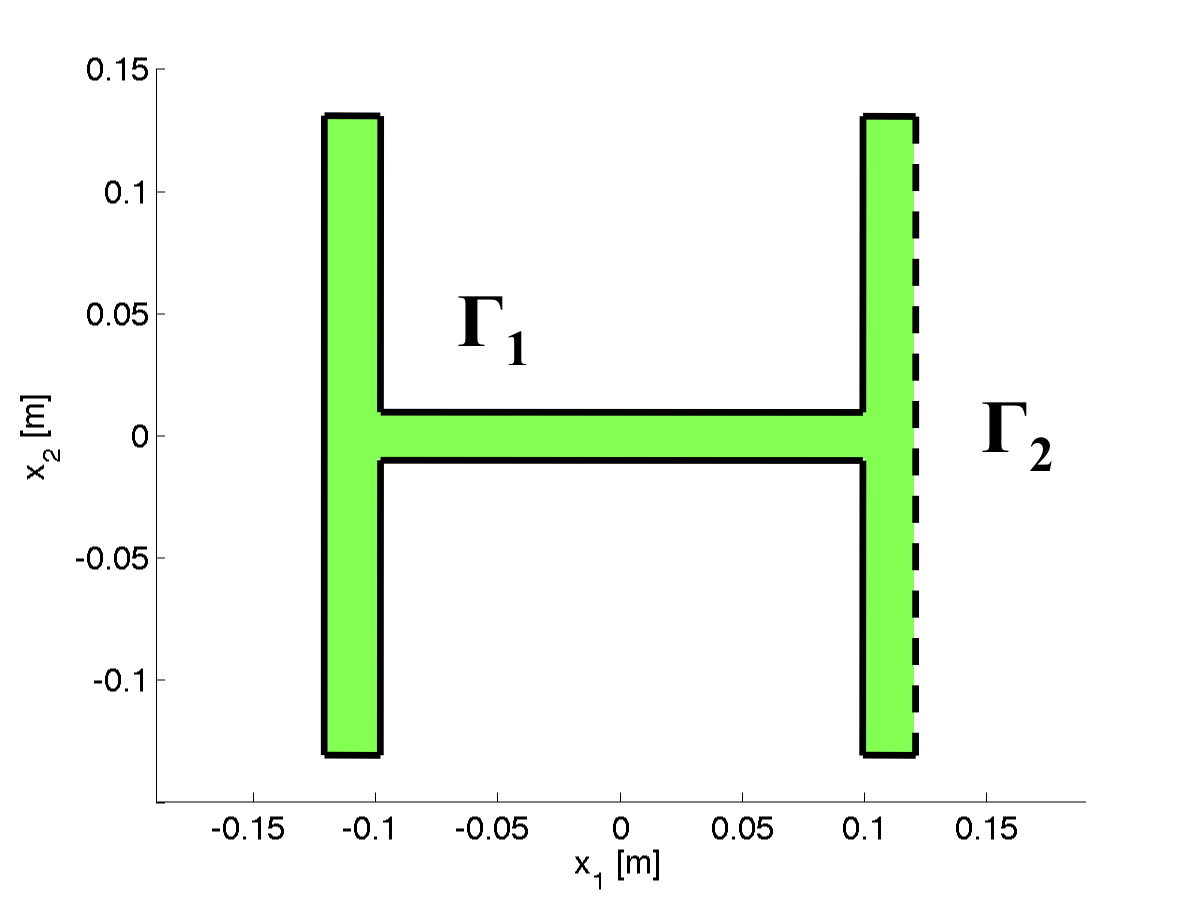}
	}\qquad
	\subfloat[Mesh]{
	\includegraphics[scale=0.28]{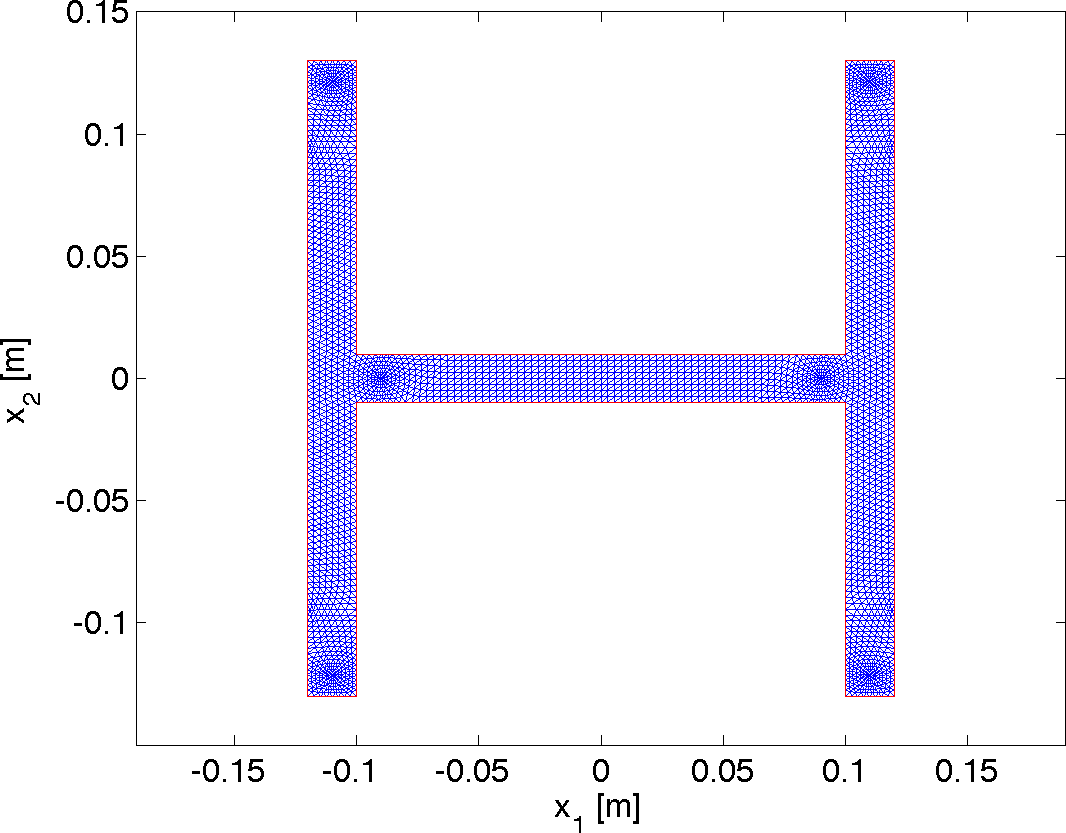}
	}
 \caption{Spatial domain $\mathcal{D}$ for the heat conduction problem \protect\eqref{eq:heat} and associated mesh.}
 \label{fig:domain}
\end{figure}

We set up the problem with the following heat conduction model. For 
$x=(x_1,x_2)\in\mathcal{D}$ and $t\in[0,70]$, let $T=T(x,t,s)$ be the time and space dependent
temperature distribution that satisfies the heat conduction model
\begin{equation}
\label{eq:heat}
\rho c\frac{\partial T}{\partial t} = -\nabla\cdot(\kappa\nabla T),
\end{equation}
with boundary conditions
\begin{equation}
T = T_b(x,t), \quad x\in\Gamma_1, \;\mbox{ and }\; -\kappa\nabla T = 0, \quad x\in\Gamma_2
\end{equation}
where
\begin{equation}
\label{eq:boundary}
T_b(x,t) = \min\left\{\,1100,\, \max\left[\,20,\, \frac{98}{3}t-6000x-700\,\right]\,\right\}.
\end{equation}
The space and time dependent Dirichlet boundary condition represents a rapidly warming environment; its spatial
dependence is plotted for three different times in Figure \ref{fig:boundary}. Notice that the temperature dependent
thermal conductivity $\kappa$ makes the model nonlinear. We set the initial condition to $T=20^{\circ}\mathrm{C}$
throughout the domain.

The spatial domain $\mathcal{D}$ is discretized with 2985 nodes on the irregular triangular mesh shown in Figure
\ref{fig:domain}, and the solution is approximated in space with standard piecewise linear finite elements; all spatial
discretization is performed with the MATLAB PDE Toolbox. The time stepping is performed by MATLAB's \texttt{ode15s}
time integrator on the spatially semidiscrete form of (\ref{eq:heat}). For reference, the temperature distribution at
three points in time is plotted in figure \ref{fig:heat} using the mean value $\kappa=\exp(\bar{Y})$.

\begin{figure}
\centering
\subfloat[]{
\includegraphics[scale=0.20]{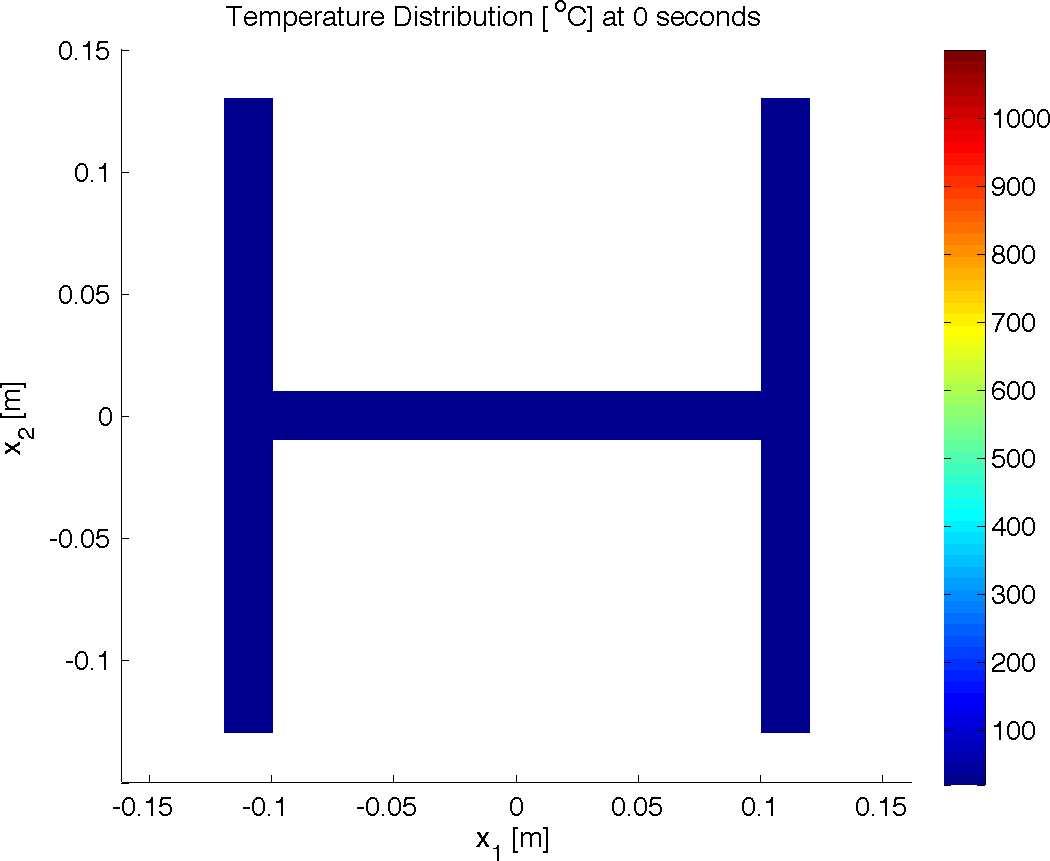}
}\quad
\subfloat[]{
\includegraphics[scale=0.20]{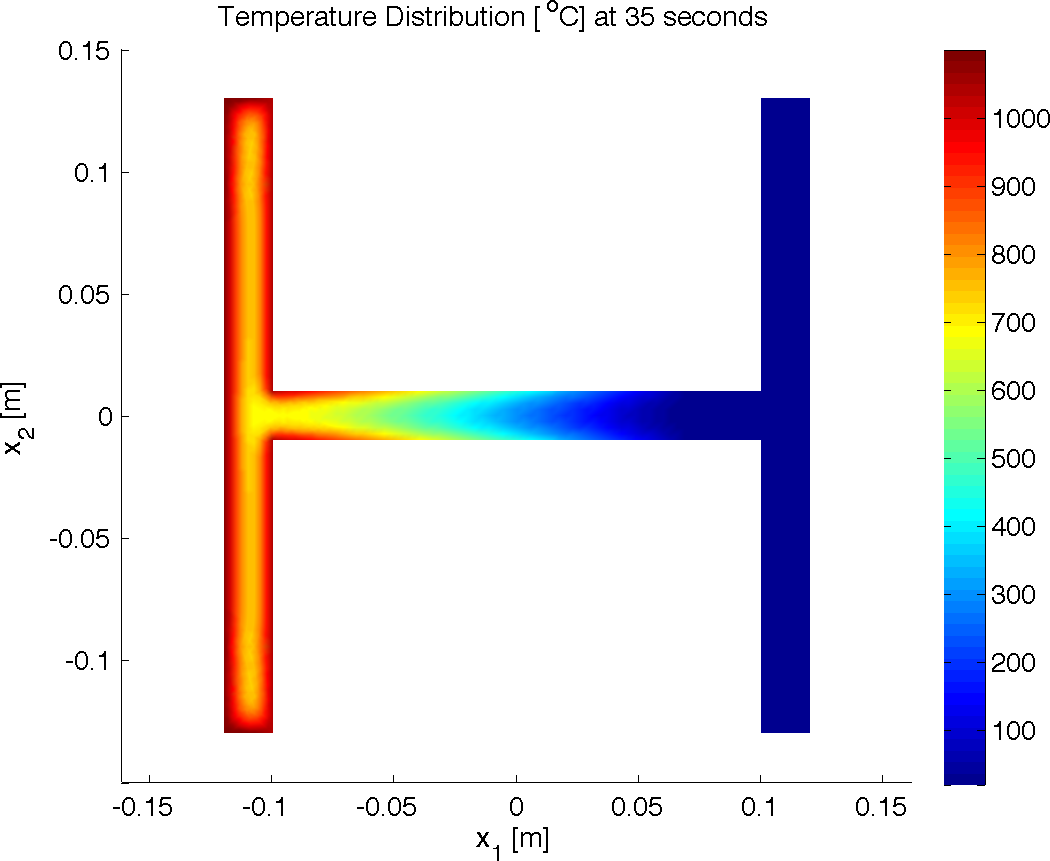}
}\quad
\subfloat[]{
\includegraphics[scale=0.20]{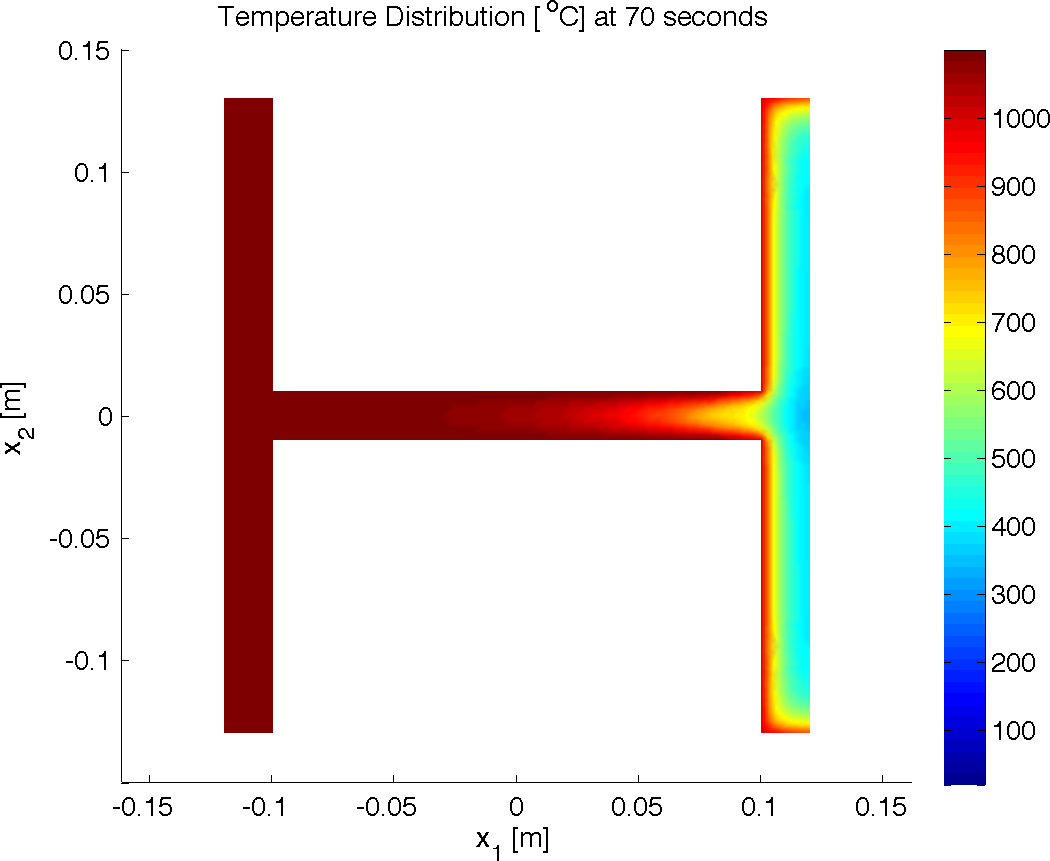}
}
\caption{Temperature distribution at the three different times using the mean trend for $\kappa$.}
\label{fig:heat}
\end{figure}

\begin{figure}
\centering
\subfloat[]{
\includegraphics[scale=0.20]{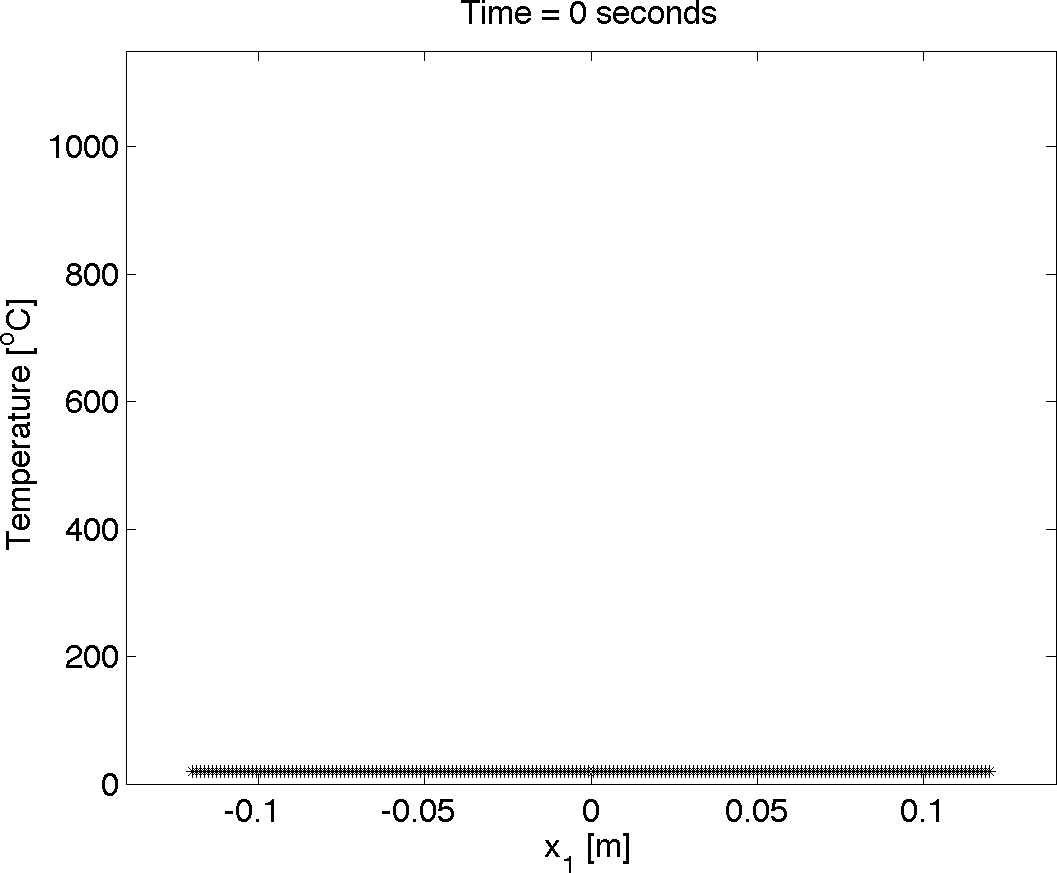}
}\quad
\subfloat[]{
\includegraphics[scale=0.20]{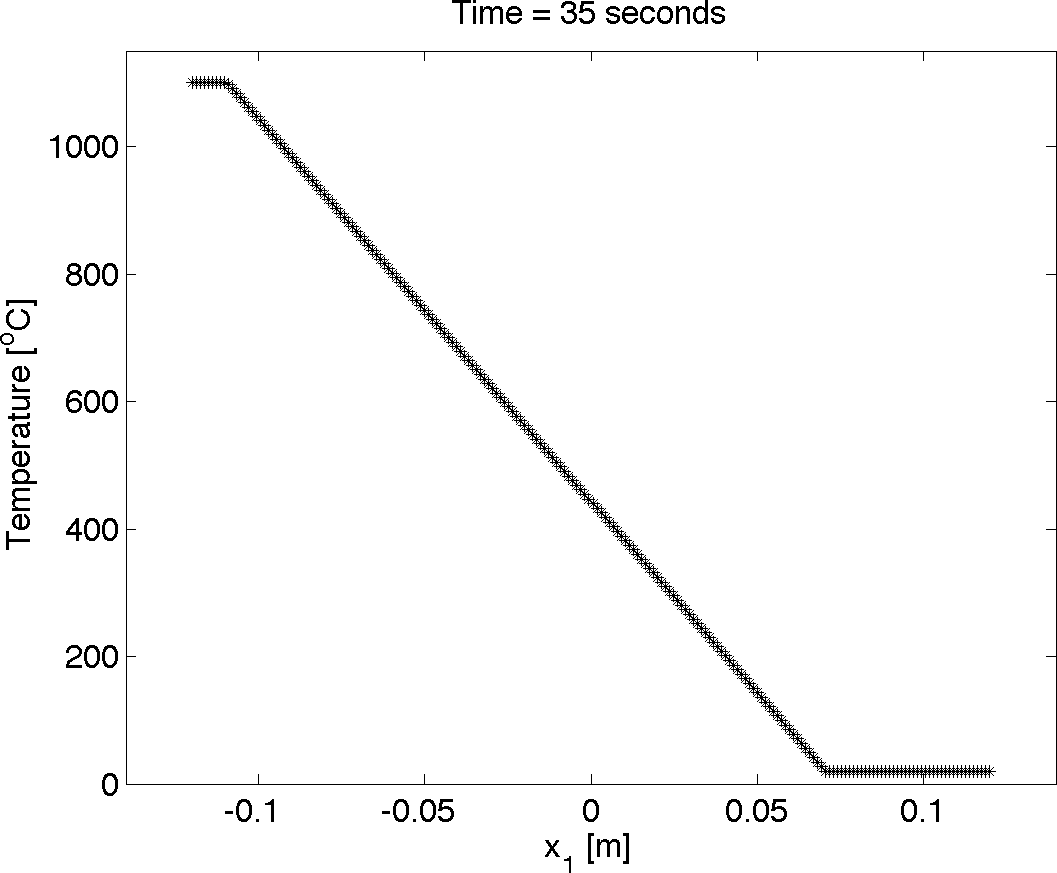}
}\quad
\subfloat[]{
\includegraphics[scale=0.20]{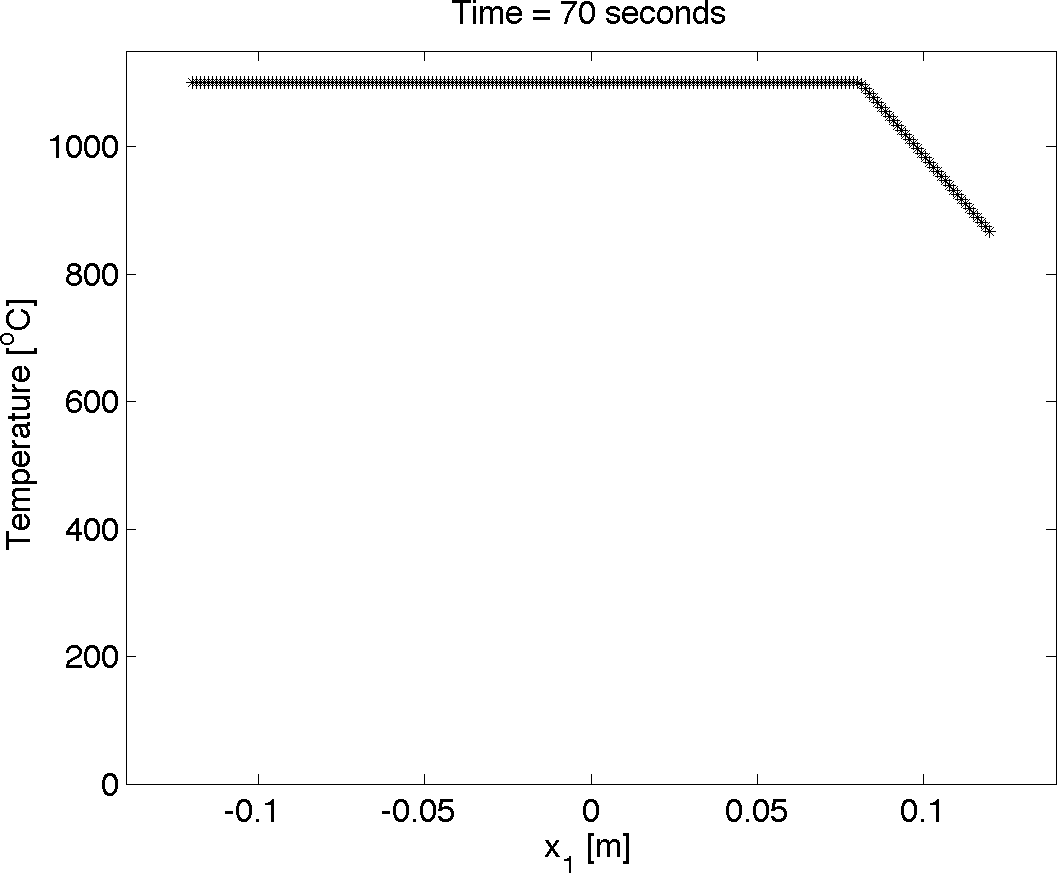}
}
 \caption{Space and time dependent Dirichlet boundary condition on the boundary segment $\Gamma_1$.}
 \label{fig:boundary}
\end{figure}

\subsubsection{Model Interpolation Study}

To test the residual minimizing model interpolation scheme, we set up the following experiment. We first draw 1000
independent realizations of a standard Gaussian random vector with $d=11$ independent components. Denote these points
$s_k\in\sS$ with $k=1,\dots,1000$ where $\sS$ is the input parameter space. For each $s_k$, we compute the
corresponding $\kappa(T,s_k)$, and the resulting temperature distribution $T_k=T(x,t,s_k)$ with the Matlab solver. From
each temperature distribution, we compute the fraction of the distribution that exceeds a critical threshold
$\tau=1000^{\circ}\mathrm{C}$ at time $t=70$ seconds. Let $Q_k=Q(s_k)$ be defined by the numerical approximation
\begin{equation}
Q_k \;\approx\; \frac{1}{|\mathcal{D}|}\int_{\mathcal{D}} I(\,T_k > \tau\,) \,dx.
\end{equation}
These $Q_k$ will constitute the cross-validation data set. 

To construct the interpolant, we draw 20 independent realizations of an 11-dimensional standard Gaussian random vector;
denote these points by $s_j\in\sS$. For each $s_j$ with $j=1,\dots,20$, we compute $T_j=T(x,t,s_j)$. From each
time history of the temperature distribution we retain the distribution at the final time $t=70$ seconds; these
constitute the basis elements for $t=70$. We also compute the quantity $f(T_j)=-\nabla\cdot(\kappa \nabla T_j)$ at time
$t=70$ for each basis element, which is used to set up the nonlinear least squares problems \eqref{eq:nlls}.

For each $s_k$ from the cross-validation set, we build the interpolant $\tilde{T}_k=\tilde{T}(x,t=70,s_k)$ using
the 20 basis elements. We then compute the fraction of the interpolated temperature distribution that exceeds
the threshold,
\begin{equation}
\tilde{Q}_k \;\approx\; \frac{1}{|\mathcal{D}|}\int_{\mathcal{D}} I(\,\tilde{T}_k > \tau\,) \,dx.
\end{equation}
We compute the error $\mathcal{E}_k=\mathcal{E}(s_k)$ with respect to the cross-validation data
\begin{equation}
\label{eq:errcomp}
\mathcal{E}_k = |Q_k-\tilde{Q}_k|.
\end{equation}
Errors are averaged over the $s_k$, $\mathcal{E}=(1/1000)\sum_k \mathcal{E}_k$. 

To test the windowing heuristic from Section \ref{sec:cred}, we choose the $M=5$ basis elements nearest the
interpolation point $s_k$ from each basis set. We compute the same $\tilde{Q}_k$ for each $s_k$ using the smaller basis
set, and the error is computed as in \eqref{eq:errcomp}. 

To average out some of the effects of randomly choosing an especially good or bad basis set with respect to the
cross-validation set, we repeat this experiment 10 times with different randomly chosen basis sets. Errors are averaged
over all $s_k$ in the cross-validation set and over all 10 expereiments. A histogram of the log of all errors is shown
in figure \ref{fig:hists} for both the full basis set and the windowed basis set. We see that there is practically no
difference in error for the smaller windowed basis set.

\begin{figure}
\centering
\subfloat[Full interpolation error]{
\includegraphics[scale=0.32]{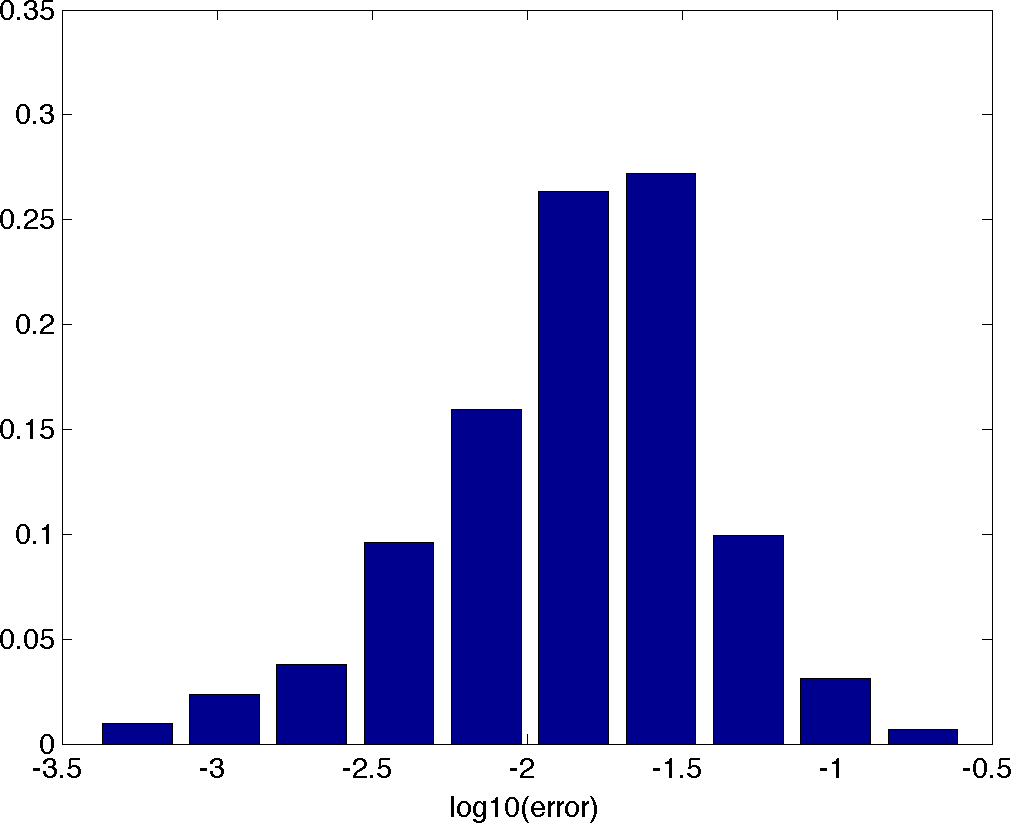}
}
\subfloat[Reduced interpolation error]{
\includegraphics[scale=0.32]{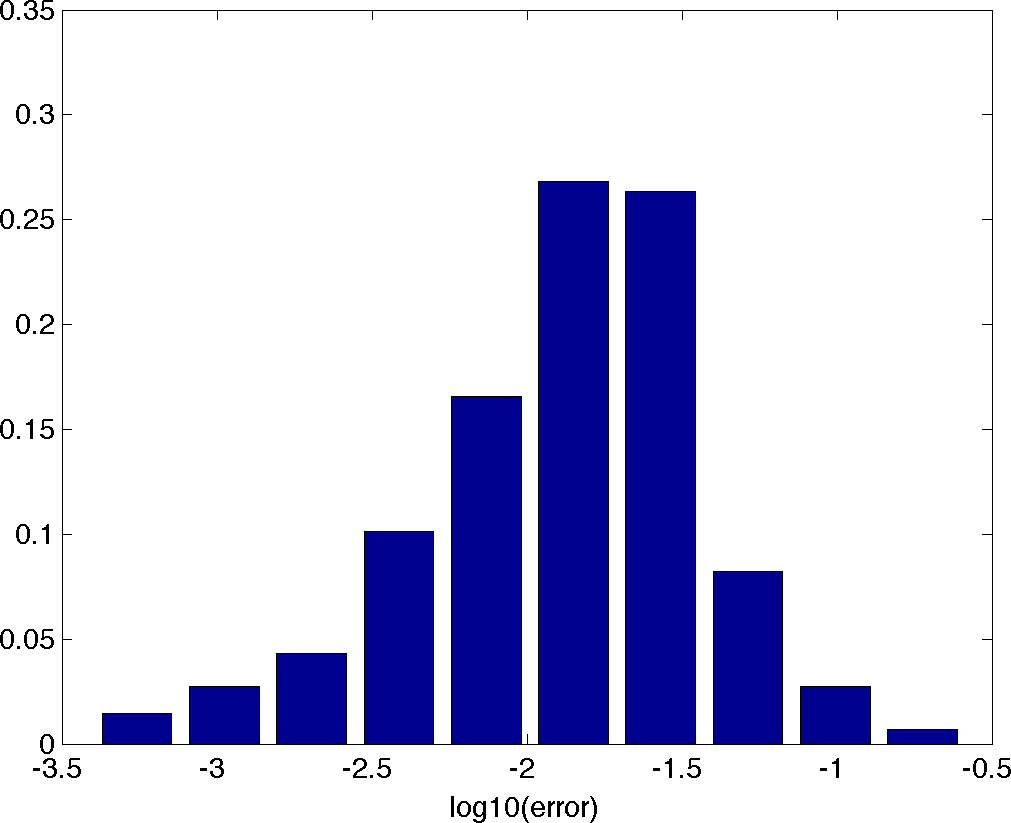}
}
\caption{Histograms of the log of the error between the interpolants and the cross-validation data.}
\label{fig:hists}
\end{figure}

In figure \ref{fig:timez}, we show a histogram of the wall clock timings for the full basis set and the smaller basis
sets. We see that the basis reduction cuts wall clock time by 24\% with no practical
change in error. In table \ref{tab:ex1}, we display the average and standard deviations of the wall clock timings for
computing the full model (the cross-validation data), the full interpolation (all 20 bases), and the reduced
interpolant (5 chosen bases). We also show the average number of function evaluations in each case; the counts of
function evaluations for the full model are output by the Matlab \texttt{ode15s} solver. For the full and reduced
interpolants, the nonlinear least squares solver used a maximum of ten Newton iterations for each evaluation.

\begin{figure}
\centering
\includegraphics[scale=0.32]{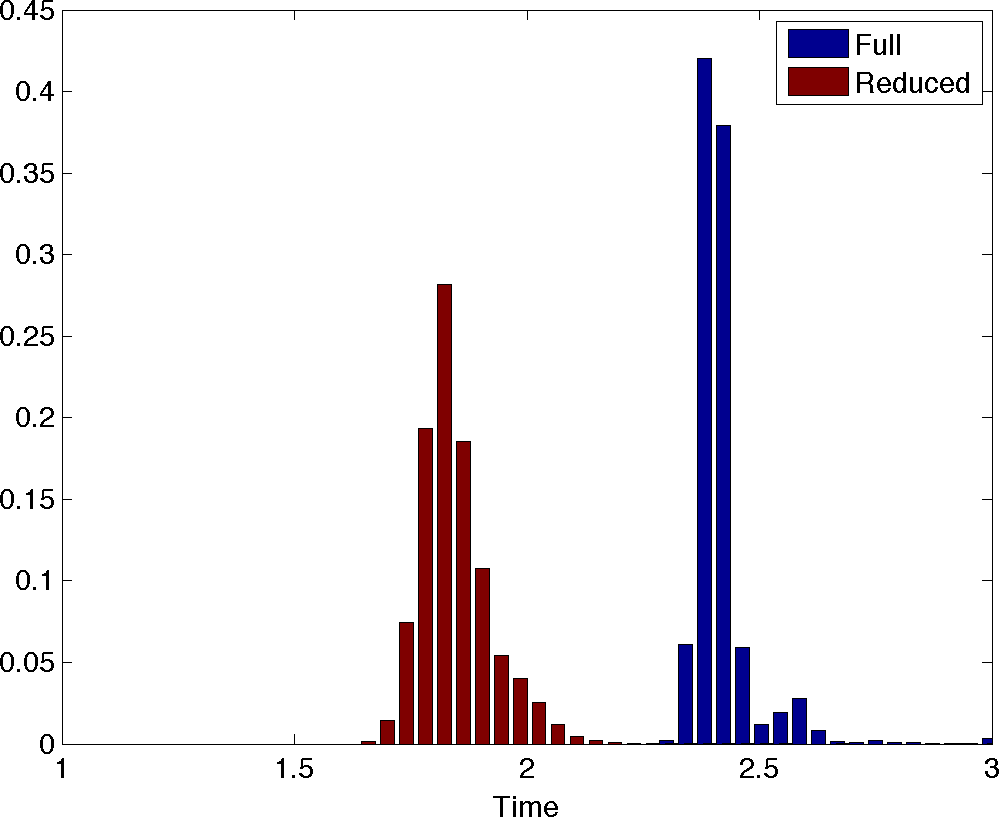}
\caption{Histogram of the wall clock times for the full and reduced interpolants.}
\label{fig:timez}
\end{figure}

\begin{table}
\caption{Timing and function evaluation counts for full model and interpolation.}
\begin{center}
\begin{tabular}{|c|c|c|c|}
\hline
 & Avg. Time (s) & Std. Time (s) & Avg. \# $f$ Evals \\
\hline
Full Model & 178.82 & 5.59 & 6166 \\
Full Interp. & 2.42 &  0.09 & 230 \\
Red. Interp. & 1.84 & 0.08 & 65 \\
\hline
\end{tabular}
\end{center}
\label{tab:ex1}
\end{table}

\section{Conclusions}
\label{sec:conclusions}

We have presented a method for approximating the solution of a parameterized, nonlinear dynamical system
using an affine combination of the time histories computed at other input parameter values. The coefficients
of the affine combination are computed with a nonlinear least squares procedure that minimizes the residual of the
governing equations. In many cases of interest, the computational cost is less than evaluating the full
model, which suggests use for reduced order modeling. This residual minimizing scheme has similar
error and convergence properties to existing reduced basis methods and POD-Galerkin reduced order models; it stands out
from existing methods for its ease of implementation by requiring only independent evaluations of the forcing function
of the dynamical system. Also, since we do not reduce the basis with a POD type reduction, the approximation
interpolates the true time history at the parameter values corresponding to the precomputed solutions.

We proved some interesting properties of this scheme including continuity, convergence, and a lower bound on the error,
and we also show that the value of the minimum residual is intimately tied to the conditioning of the least squares
problems used to compute the Newton steps. We introduced heuristics to combat this ill-conditioning and further reduce
the cost of the method, which we tested with two numerical examples: (i) a three-state dynamical system representing
kinetics with a single parameter controlling stiffness and (ii) a nonlinear, parabolic PDE with a high-dimensional
random conductivity field.

\section{Acknowledgments}

The authors thank David Gleich at Purdue University for many helpful conversations and insightful comments. We also
thank Venkatesh Kodur at Michigan State University for providing us with the data to motivate and construct the
stochastic model of thermal conductivity. We also thank the anonymous reviewers for their helpful and insightful
comments.

\bibliographystyle{siam}
\bibliography{paulconstantine}

\end{document}